%% file: main_arxiv_version.tex
\documentclass[11pt,letterpaper]{article}
\usepackage{booktabs}
\usepackage[letterpaper,margin=1in]{geometry}

\usepackage{bm}
\usepackage{mlmodern}
\usepackage{inconsolata}

\usepackage{booktabs}
\usepackage{mdframed}

\usepackage{dsfont}

\usepackage{amsfonts,nicefrac}
\usepackage[utf8]{inputenc}
\usepackage[dvipsnames]{xcolor}
\usepackage{graphicx}
\usepackage{multirow}
\usepackage{amsmath}
\usepackage{amsthm}
\usepackage{amssymb}
\usepackage{comment}
\usepackage{cancel}

\usepackage[shortlabels]{enumitem}
\usepackage[style=alphabetic, maxbibnames=15, giveninits=true, maxcitenames=15, natbib=true,
  maxalphanames=10, backend=biber, sorting=nty, backref=true, uniquename=false]{biblatex}
\DefineBibliographyStrings{english}{backrefpage = {page},backrefpages = {pages},}
\DeclareNameAlias{sortname}{given-family}

\makeatletter
\newtheorem{theorem}{Theorem}[section]
\newtheorem{corollary}{Corollary}[theorem]
\newtheorem{lemma}[theorem]{Lemma}
\newtheorem{proposition}[theorem]{Proposition}
\newtheorem{definition}{Definition}[section]
\newtheorem{assumption}{Assumption}[section]
\theoremstyle{remark}
\newtheorem{remark}{Remark}[section]

\usepackage{bm}
\usepackage{hyperref}
\hypersetup{
  colorlinks   = true, %
  urlcolor     = blue, %
  linkcolor    = blue, %
  citecolor   = red %
}
\input{math_commands.tex}

\makeatother

\newcommand{\red}[1]{{\textcolor{red}{#1}}}

\renewcommand{\epsilon}{\varepsilon}

\usepackage{algorithm}
\usepackage{algpseudocode}

\floatstyle{ruled}
\newfloat{subroutine}{htbp}{lok}
\floatname{subroutine}{Subroutine}
\usepackage{tikz}
\usetikzlibrary{decorations.pathreplacing,calc}
\newcommand{\tikzmark}[1]{\tikz[overlay,remember picture] \node (#1) {};}

\newcommand*{\AddNote}[4]{%
    \begin{tikzpicture}[overlay, remember picture]
        \draw [decoration={brace,amplitude=0.3em},decorate,thick,black]
            ($(#3.east)!(#1.north)!($(#3.east)-(0,1)$)$) --  
            ($(#3.east)!(#2.south)!($(#3.east)-(0,1)$)$)
                node [align=left, text width=4cm, pos=0.5, anchor=west] {#4};
    \end{tikzpicture}
}
\definecolor{comment}{RGB}{2,128, 9}

\usepackage{eqparbox}

\usepackage[capitalise]{cleveref}
\crefname{equation}{}{}
\newcommand{\myalert}[1]{{\vspace{2mm}\noindent\textbf{#1}}}
\newcommand{\reals}{\mathbb{R}}

\newcommand{\vDelta}{\bm{\Delta}}

\newcommand{\Reg}{\mathrm{Reg}}

\newcounter{boxcounter}

\begin{document}

\title{%
Improved Complexity for Smooth Nonconvex Optimization: A Two-Level Online Learning Approach with Quasi-Newton Methods\footnotetext{\llap{\textsuperscript{1}}The authors are listed in alphabetical order.}
}

\author{Ruichen Jiang\thanks{Department of Electrical and Computer Engineering, The University of Texas at Austin, Austin, TX, USA  \{rjiang@utexas.edu, mokhtari@austin.utexas.edu, fpatitucci@utexas.edu\}} \and Aryan Mokhtari$^*$ \and Francisco Patitucci$^*$}

\date{}

\maketitle

\begin{abstract}%
We study the problem of finding an \(\epsilon\)-first-order stationary point (FOSP) of a smooth function, given access only to gradient information. 
The best-known gradient query complexity for this task, assuming both the gradient and Hessian of the objective function are Lipschitz continuous, is \(\mathcal{O}(\epsilon^{-7/4})\).
In this work, we propose a method with a gradient complexity of \(\mathcal{O}(d^{1/4}\epsilon^{-13/8})\), where \(d\) is the problem dimension, leading to an improved complexity when \(d = \mathcal{O}(\epsilon^{-1/2})\).
To achieve this result, we design an optimization algorithm that, underneath, involves solving two online learning problems. Specifically, we first reformulate the task of finding a stationary point for a nonconvex problem as minimizing the regret in an online convex optimization problem, where the loss is determined by the gradient of the objective function. Then, we introduce a novel \textit{optimistic quasi-Newton  method} to solve this online learning problem, with the Hessian approximation update itself framed as an online learning problem in the space of matrices. 
Beyond improving the complexity bound for achieving an \(\epsilon\)-FOSP using a gradient oracle, our result provides the first guarantee suggesting that quasi-Newton methods can potentially outperform gradient descent-type methods in nonconvex settings.
\end{abstract}

\newpage

\input{body_v2}

\section*{Acknowledgments}

{This work is supported in part by NSF Grant CCF-2007668 and the NSF AI Institute for Foundations of Machine Learning (IFML). The authors would like to thank Ashok Cutkosky and Francesco Orabona for their helpful comments on the first draft of the paper.}

\newpage
\appendix
\input{appendix}

\printbibliography
\end{document}

%% file: math_commands.tex
\newcommand\Vector[1]{\mathbf{#1}}

\newcommand\vb{{\Vector{b}}}

\newcommand\ve{{\Vector{e}}}

\newcommand\vg{{\Vector{g}}}
\newcommand\vh{{\Vector{h}}}

\newcommand\vs{{\Vector{s}}}

\newcommand\vu{{\Vector{u}}}
\newcommand\vv{{\Vector{v}}}
\newcommand\vw{{\Vector{w}}}
\newcommand\vx{{\Vector{x}}}
\newcommand\vy{{\Vector{y}}}
\newcommand\vz{{\Vector{z}}}

\newcommand\MATRIX[1]{\mathbf{#1}}

\newcommand\mA{{\MATRIX{A}}}
\newcommand\mB{{\MATRIX{B}}}

\newcommand\mG{{\MATRIX{G}}}
\newcommand\mH{{\MATRIX{H}}}
\newcommand\mI{{\MATRIX{I}}}

\newcommand\mM{{\MATRIX{M}}}

\newcommand\mS{{\MATRIX{S}}}
\newcommand\mT{{\MATRIX{T}}}

\newcommand\mV{{\MATRIX{V}}}
\newcommand\mW{{\MATRIX{W}}}

\newcommand\sS{{\mathbb{S}}}

\newcommand\bigO{\mathcal{O}}

\DeclareMathOperator*{\E}{\mathbb{E}}
\newcommand\op{{\mathrm{op}}}

\DeclareMathOperator*{\argmax}{arg\,max}
\DeclareMathOperator*{\argmin}{arg\,min}

%% file: body_v2.tex
\section{Introduction}

In this paper, we address the problem of finding a near-stationary point of a smooth, nonconvex function \( f \). When the gradient of \( f \) is Lipschitz continuous, it is known that gradient descent can find an \(\varepsilon\)-first-order stationary point (FOSP)—where \(\|\nabla f(\mathbf{x})\|_2 \leq \varepsilon\)—in at most \(\mathcal{O}(\varepsilon^{-2})\) iterations. Furthermore, with only first-order information and a Lipschitz continuous gradient, this complexity is optimal and matches the established lower bound in~\cite{carmon2020lower}.

\looseness=-1
Interestingly, even with the same oracle, where only the function's gradient is available, additional assumptions can reduce the complexity for finding an \(\varepsilon\)-FOSP. Specifically, \cite{carmon2017convex} introduced an accelerated gradient descent variant that leverages negative curvature to reach an \(\varepsilon\)-FOSP in at most \(\mathcal{O}(\varepsilon^{-7/4}\log(1/\varepsilon))\) gradient queries, assuming that both the gradient and Hessian of the objective function are Lipschitz continuous.
They further showed that if the third derivative is also Lipschitz, the number of gradient queries can be reduced to \(\mathcal{O}(\varepsilon^{-5/3}\log(1/\varepsilon) )\). In concurrent work, \cite{agarwal2017finding} achieved a similar complexity under the assumption that both the gradient and Hessian are Lipschitz continuous. Notably, they introduced a variant of the cubic regularization Newton method~\cite{nesterov2006cubic} that requires access only to the gradient of the objective function {and Hessian-vector products}. Their method finds an \(\varepsilon\)-FOSP using \(\mathcal{{O}}(\varepsilon^{-7/4}\log(d/\epsilon))\) {Hessian-vector products}. %

In follow-up work, \cite{li2022restarted,li2023restarted} successfully removed the polylogarithmic factor from the complexity of the previous results by introducing a restarted variant of the accelerated gradient descent method and the heavy ball method. Specifically, they demonstrated that it is possible to achieve an \(\varepsilon\)-FOSP with \(\mathcal{O}(\varepsilon^{-7/4})\) gradient queries, assuming both the gradient and Hessian are Lipschitz continuous. Later, the authors in \cite{Marumo2024,Marumo2024a} further developed parameter-free methods by incorporating line search, thus removing the need for prior knowledge of problem parameters.

It is also worth mentioning that several studies have investigated the problem of finding a \textit{second-order stationary point}, which is a more difficult task~\cite{agarwal2017finding,carmon2018accelerated,Jin2018,allenzhu2018neon2,Xu2017,royer2018complexity,Royer2020}. Although these methods also yield an \(\varepsilon\)-FOSP as a byproduct, none achieves a gradient complexity better than \(\mathcal{O}(\varepsilon^{-7/4})\) to find an \(\varepsilon\)-FOSP.

\textbf{Contributions.} 
Our main contribution is breaking the existing \(\mathcal{O}(\varepsilon^{-7/4})\) complexity barrier using only gradient oracles, assuming Lipschitz continuity of both the gradient and Hessian. 
We achieve this goal by proposing an optimization
method that integrates a two-level online learning approach.
At the first level, inspired by \cite{cutkosky2023optimal}, we reformulate the task of finding a first-order stationary point for a nonconvex function as an online convex optimization problem, where the loss is defined by the gradient of the objective function. We introduce a novel \textit{optimistic quasi-Newton method} to address this online learning problem. Guided by our convergence analysis, the update of the Hessian approximation in this quasi-Newton method naturally leads to a second online learning problem, framed in the space of matrices with a quadratic loss. This approach enables our method to rely exclusively on gradient queries, eliminating the need for any second-order information, including Hessian-vector products.
We establish that our method  achieves a gradient complexity of $\bigO(d^{1/4}\epsilon^{-13/8})$ for finding an $\epsilon$-FOSP, which outperforms the best existing complexity of \(\mathcal{O}(\varepsilon^{-7/4})\) when the problem dimension satisfies $d=\mathcal{O}(\varepsilon^{-1/2})$. Moreover, we show that the total number of matrix-vector products required by our algorithm is bounded by $\tilde{\bigO}({d^{1/8}}{\epsilon^{-29/16}}+d^{3/8}\epsilon^{-27/16})$.     

\subsection{Additional Related Work} 

\myalert{Quasi-Newton methods in nonconvex settings.} Quasi-Newton methods, widely used for unconstrained minimization, include popular updates like DFP~\cite{davidon1959variable,fletcher1963rapidly}, BFGS~\cite{broyden1970convergence,fletcher1970new,goldfarb1970family,shanno1970conditioning}, and SR1~\cite{davidon1959variable,conn1991convergence,khalfan1993theoretical}. Despite their practical success, convergence properties for these methods have been established primarily for strongly convex or convex functions~\cite{powell1971convergence,broyden1973local,powell1976some,byrd1987global}. For decades, the global convergence of BFGS on nonconvex objectives remained an open question.~\cite{Nocedal1992Theory,fletcher1994overview}. 
While pathological examples show that  BFGS update may fail to converge~\cite{dai2002convergence,mascarenhas2004bfgs}, the authors in \cite{li2001modified,li2001global} established that, with regularization or a skipping mechanism, the iterates generated by BFGS satisfy $\lim \inf_{t \rightarrow \infty}\|\nabla f(\vx_t)\| = 0$. However, these results only show asymptotic convergence and no explicit convergence rate was given. To our knowledge, no theoretical results have yet demonstrated a provable advantage for quasi-Newton methods in the nonconvex setting. An additional contribution of our result is to provide the first guarantee that a quasi-Newton method can outperform gradient descent-based methods in finding a first-order stationary point of a nonconvex function.

\section{Preliminaries and Background} \label{sec:prelims}

Formally, we consider the unconstrained minimization problem:
\begin{equation}\label{eq:min}
  \min_{\vx\in \reals^d} f(\vx),
\end{equation}
where $f : \reals^d \rightarrow \reals$ is smooth but possibly nonconvex. We assume that $f(\vx)$ is bounded below with an optimal value $f^*$ and satisfies the following two assumptions. Unless otherwise specified, we use $\|\cdot\|$ to denote the $\ell_2$-norm for vectors and the operator norm of matrices. 

\begin{assumption}[Lipschitz gradient]\label{asm:L1} 
   $\|\nabla f(\vx) - \nabla f(\vy)\| \leq L_1 \|\vx-\vy\|$ for any $\vx, \vy \in \reals^d$.
\end{assumption}

\begin{assumption}[Lipschitz Hessian]\label{asm:L2}
    $\|\nabla^2 f(\vx) - \nabla^2 f(\vy)\| \leq L_2 \|\vx-\vy\|$ for any $\vx, \vy \in \reals^d$. 
\end{assumption}

\subsection{Online-to-nonconvex Conversion
}
To lay the groundwork for our algorithm, we first review the approach introduced by~\cite{cutkosky2023optimal}, which converts the problem of finding a stationary point of a function $f$ into solving an online learning problem. Specifically, consider the general update rule
$
\vx_{n}=\vx_{n-1} + \vDelta_n
$, 
where we assume that $\|\vDelta_n\| \leq D$. Rather than directly prescribing the update rule, the key idea in ~\cite{cutkosky2023optimal} is to let the convergence analysis guide our choice of $\vDelta_n$.

Note that given Assumption~\ref{asm:L1}, one can show that  $
     f(\vx_{n-1}) - f(\vx_n) \geq -\nabla f (\vx_{n-1})^\top \vDelta_n - \frac{L_1}{2} D^2$.
Hence, 
to maximize the function value decrease, $\vDelta_n = - D \frac{\nabla f (\vx_{n-1})}{\|\nabla f (\vx_{n-1})\|}$ is the best choice and recovers the update of (normalized) gradient descent. Following standard analysis and a suitable choice of $D$, this yields a complexity of $\bigO(1/\varepsilon^2)$.
The main observation in \cite{cutkosky2023optimal} 
is that we can derive a tighter lower bound on the function value decrease by either leveraging randomization or applying Assumption~\ref{asm:L2}. Specifically, if $\vg_n$ is the gradient at a random point along the segment between $\vx_{n-1}$ and $\vx_n$, the change in the function value, $f(\vx_n) - f(\vx_{n-1})$, is exactly equal to $\mathbb{E}[\vg_n^\top \vDelta_n]$. 
Alternatively, as in \cite[Section 6.1]{cutkosky2023optimal}, one can set $\vg_n$ as the gradient at the midpoint between $\vx_{n-1}$ and $\vx_n$ to remove randomness from the analysis, which we adopt in this paper. Specifically, define $\vw_n=\frac{1}{2}(\vx_{n-1} + \vx_n)$ and set $\vg_n=\nabla f(\vw_n)$. This modification introduces an error in approximating $f(\vx_n) - f(\vx_{n-1})$ by $\vg_n^\top \vDelta_n$, but we show in the next lemma that this error is negligible when $D$ is small under Assumption~\ref{asm:L2}. {The proof is in Appendix~\ref{appen:conversion_1_s}.}
\begin{lemma}\label{lem:conversion}
    Consider $\vx_{n}=\vx_{n-1} + \vDelta_n$ where  $\|\vDelta_n\| \leq D$. Further, define $\vg_n=\nabla f(\vw_n)$  where $\vw_n=\frac{1}{2}(\vx_{n-1}+\vx_{n})$. 
    If Assumption~\ref{asm:L2} holds, then  
    {$
           f(\vx_{n-1}) - f(\vx_n) \geq  -\vg_n^\top \vDelta_n - \frac{L_2D^3}{48}$}.
\end{lemma}
This selection of $\vg_n$  refines the approximation error from $\bigO(L_1 D^2)$ to $\bigO(L_2 D^3)$, a crucial improvement for achieving better complexity. 
By Lemma~\ref{lem:conversion}, the optimal choice for $\vDelta_n$ is \(-\vg_n\). However, computing \(\vg_n\) at \(\vw_n\) requires \(\vx_n\), which is unavailable when selecting \(\vDelta_n\).
This point suggests that choosing \(\vDelta_n\) to maximize function decrease can be viewed as an \emph{online learning problem}, where \(\vDelta_n\) is the action, and the loss is the linear function \(\vg_n^\top \vDelta_n\). This perspective implies that minimizing iterations to reach a stationary point relates to an online learning problem aimed at minimizing the cumulative loss \(\sum_{n} \vg_n^\top \vDelta_n\).
To formally connect the online learning formulation of selecting $\vDelta_n$ and the function decrease to finding a stationary point, we use  Lemma~\ref{lem:conversion} stating that for any arbitrary $\vu$, after $T$ updates, we have
$
 f(\vx_T)-f(\vx_{0}) \leq \sum_{n=1}^T \vg_n^\top (\vDelta_n-\vu) +\sum_{n=1}^T \vg_n^\top \vu+\frac{TL_2D^3}{48}.
$
Now if we set  the arbitrary vector as 
$
\vu=-D \nicefrac{(\sum_{n=1}^T \vg_n)}{\|\sum_{n=1}^T \vg_n\|}
$, it can be shown:
\begin{equation}\label{eq:bound_1}
   \bigg\|\frac{1}{T}\sum_{n=1}^T \vg_n\bigg\|\leq  \frac{f(\vx_{0})- f(\vx_T)}{DT} + \frac{1}{DT}\sum_{n=1}^T \vg_n^\top (\vDelta_n-\vu) +\frac{L_2D^2}{48}. 
\end{equation}
Although the above bound connects the norm of the average gradient to the regret term $\sum_{n=1}^T \vg_n^\top (\vDelta_n - \vu)$, we aim to have the norm of the gradient at an average point on the left-hand side to guarantee stationarity. This leads us to the following lemma, similar to \cite[Proposition 15]{cutkosky2023optimal}.
\begin{lemma}\label{lemma:averaged_gradient}
   Recall the definition of $\vg_n = \nabla f(\vw_n)$. 
   If Assumption~\ref{asm:L2} holds, then 
    $
         \|\nabla f(\bar{\vw})\| \leq  \|\frac{1}{T}\sum_{n=1}^T \vg_n\| + \frac{L_2}{2}T^2 D^2$,
    where $\bar{\vw} = \frac{1}{T}\sum_{n=1}^T \vw_n$. %
\end{lemma}
Indeed, combining the above result with the expression in \eqref{eq:bound_1} connects the norm of the gradient at the average iterate to the regret bound on the right-hand side. To generalize this framework, we consider a multi-episode online learning problem. In this setting, after every $T$ iterations—referred to as the episode length—the arbitrary vector $\vu$ changes, and we reset the gradient averaging. This leads to the following proposition for the case when we have $K$ episodes.
\begin{proposition}\label{pr:average_gradient}
Suppose that Assumption~\ref{asm:L2} holds and consider Algorithm~\ref{alg:conversion}. {Define $\bar{\vw}^k = \frac{1}{T} \sum_{n=(k-1)T+1}^{kT} \vw_n$ and $\vu^k=-D \nicefrac{\sum_{n=(k-1)T+1}^{kT} \vg_n}{\|\sum_{n=(k-1)T+1}^{kT} \vg_n\|}$.} Then we have:
\begin{equation*}
        \frac{1}{K}\sum_{k=1}^K \|\nabla f(\bar{\vw}^k)\| \leq \frac{f(\vx_0) - f^*}{DKT} +\frac{1}{DKT} \Reg_T(\vu^1,\dots,\vu^K)+ \frac{L_2}{48} D^2 + \frac{L_2}{2}T^2 D^2,
    \end{equation*}    
    where $
    \Reg_T(\vu^1,\dots,\vu^K) = \sum_{k=1}^K \sum_{n=(k-1)T+1}^{kT} \langle \vg_n, \vDelta_n - \vu^k\rangle$.
\end{proposition}

Given the above discussion, any standard online learning algorithm can be employed to minimize the regret term corresponding to the presented online linear optimization problem, thereby establishing a complexity bound for reaching a stationary point of the objective function $f$. Interestingly, using a standard first-order online learning method, such as the optimistic gradient method, results in an overall complexity of $\bigO(1/\varepsilon^{7/4})$, as shown in \cite{cutkosky2023optimal}. While this framework provides a simple algorithmic scheme to recover the complexity of $\bigO(1/\varepsilon^{7/4})$, it does not improve the best-known complexity bound. In the next section, we introduce a novel optimistic quasi-Newton method, which relies solely on first-order information and achieves a regret bound that results in a better complexity bound than $\bigO(1/\varepsilon^{7/4})$ when the dimension $d$ is sufficiently small.

\section{Proposed Algorithm}\label{sec:algorithm}

\looseness=-1
In this section, we introduce our proposed method and explain its key ideas. Building on the framework in \cite{cutkosky2023optimal}, we formulate the problem of finding a stationary point of the function \( f \) as an online learning problem. In Section~\ref{subsec:optimistic}, we present a novel optimistic quasi-Newton  method to solve it. Then in Section~\ref{subsec:hessian_approx_update}, we show that the update for the Hessian approximation in our quasi-Newton algorithm boils down to solving a second online learning problem, this time in the space of matrices.

\subsection{Learning Update Directions: Optimistic Quasi-Newton Algorithm}\label{subsec:optimistic}

According to Proposition~\ref{pr:average_gradient}, 
our primary goal shifts to minimizing the $K$-shifting regret associated with this online learning formulation:
\refstepcounter{boxcounter}
\begin{mdframed}[linewidth=1pt,roundcorner=5pt]\label{box:OL1}
    \textbf{Online Learning Problem 1 %
    } \\[5pt]
    For $n=1,\dots,KT$:
    \begin{itemize}
        \item The learner chooses $\vDelta_n \in \reals^d$ such that $\|\vDelta_n\| \leq D$; 
        \item $\vg_n = \nabla f (\frac{1}{2}(\vx_{n}+\vx_{n-1}))$ is computed, where $\vx_{n}=\vx_{n-1} + \vDelta_n$; 
        \item The learner observes the loss  $\ell_n(\vDelta_n) = \langle \vg_n, \vDelta_n \rangle$;
    \end{itemize}
    \textbf{Goal:} Minimize the regret given by $\Reg_T(\vu^1,\dots,\vu^K) = \sum_{k=1}^K \sum_{n=(k-1)T+1}^{kT} \langle \vg_n, \vDelta_n - \vu^k\rangle$.
\end{mdframed}

\looseness = -1
To develop our optimistic quasi-Newton (OQN) method for addressing {Online Learning Problem~\ref{box:OL1}}, we begin by briefly reviewing the core concepts of optimistic methods~\cite{rakhlin2013online,joulani2020modular}. In our online learning setup, at each iteration \(n\), we first select an action \(\vDelta_n\) and then observe the loss \(\ell_n(\vDelta_n) = \langle \mathbf{g}_n, \vDelta_n \rangle\), where \(\mathbf{g}_n\) represents the gradient at iteration \(n\). Ideally, we would like to update \(\vDelta_n\) using the current loss gradient \(\mathbf{g}_n\), but since \(\mathbf{g}_n\) is only available after selecting~\(\vDelta_n\), we cannot use it directly.
The optimistic method addresses this challenge by employing a prediction or hint~\(\mathbf{h}_n\) to estimate~\(\mathbf{g}_n\) based on the information available up to time \(n\). We refine this estimate by correcting it with the previous estimation error, resulting in the adjusted descent direction: $
\mathbf{d}_n = \mathbf{h}_n - (\mathbf{g}_{n-1} - \mathbf{h}_{n-1})$. 
This approach is founded on the optimistic assumption that the difference between the true gradient and the hint remains consistent across iterations—that is, \(\mathbf{g}_n - \mathbf{h}_n \approx \mathbf{g}_{n-1} - \mathbf{h}_{n-1}\).
The hint \(\mathbf{h}_n\) can be any function based on the information available up to time \(n\). Note that for {Online Learning Problem~\ref{box:OL1}},
the general update of optimistic method can be written as:
\begin{equation}\label{eq:optimistic_c3}
     \vDelta_{n} = \Pi_{\|\vDelta\|\leq D} \left(\vDelta_{n-1} - \eta \vh_{n} - \eta (\vg_{n-1} - \vh_{n-1})\right), \quad \forall n >1 ,
\end{equation}
while for $n=1$ it is simply $\vDelta_{1} = \Pi_{\|\vDelta\|\leq D} \left(\vDelta_{0} - \eta \vh_{1} \right)$. 
With access to only first-order oracle, a natural choice is to set \(\mathbf{h}_n = \mathbf{g}_{n-1}\) and this will lead to a complexity of $O(1/\varepsilon^{2})$. However, due to the structure of the problem, one can construct another hint $\mathbf{h}_n$ that better approximates $\mathbf{g}_n$. Recall that $\vg_n = \nabla f(\vw_n)$ and $\vw_n = \vx_{n-1}+\frac{1}{2}\vDelta_n$. Specifically, based on the assumption that $\vDelta_{n} \approx \vDelta_{n-1}$, one can define $\vz_{n-1} = \vx_{n-1}+\frac{1}{2}\vDelta_{n-1}$ and set $\vh_n = \nabla f(\vz_{n-1})$, which would lead to a complexity of $O(1/\varepsilon^{7/4})$. While this hint improves the complexity of $O(1/\varepsilon^{2})$, it fails to show any improvement over the best-known bound.
To overcome this issue, we propose a quasi-Newton optimistic method that, while utilizing only first-order information, manages to provide a hint  better than the choice of {\(\vh_n = \nabla f(\vz_{n-1})\)} used in the  optimistic gradient method. 

\begin{algorithm}[!t]\small
    \caption{Optimistic Quasi-Newton for Online-to-nonconvex Conversion}\label{alg:conversion}
    \begin{algorithmic}[1]
    \Require Initial point $\vx_0$, initial matrix $\mB_1$ s.t. $\|\mB_1\|_{\op} \!\leq\! L_1$, $K,T \in \mathbb{N}$, 
    radius $D$, subproblem accuracy $\delta$
    \item[\textbf{Initialize:}] $\vDelta_1 = -D \frac{\nabla f(\vx_0)}{\|\nabla f(\vx_0)\|}$, $\vh_1 = \nabla f(\vx_0)$
    \For{$n = 1$ to $K T$}
        \State Set $\vx_n = \vx_{n-1} + \vDelta_n$
        \State Set $\vw_n = \vx_{n-1} + \frac{1}{2}\vDelta_n$, $\vg_n = \nabla f(\vw_n)$, $\vz_n = \vx_n+\frac{1}{2}\vDelta_n$
        \State Set $\vDelta_{n+1} = \mathsf{TRSolver}(\mA_n, \vb_n, D, \delta)$, where $\mA_n = \frac{1}{2}\mB_n + \frac{1}{\eta} \mI$, $\vb_n =\nabla f(\vz_n) + \vg_n - \vh_{n}- \frac{1}{2}\mB_n \vDelta_n - \frac{1}{\eta}\vDelta_n$
        \State Set $\vh_{n+1} = \nabla f(\vz_n) + \frac{1}{2}\mB_n (\vDelta_{n+1} - \vDelta_{n})$ \label{line:TRcall}
        \State \Comment{We have $\vDelta_{n+1} \approx \Pi_{\|\vDelta\|\leq D} \left(\vDelta_n - \eta \vh_{n+1} - \eta (\vg_n - \vh_n)\right)$; See Section~\ref{subsec:optimistic} %
        }
        \State Set $\vy_n = {\vg_{n+1}-\nabla f(\vz_n)}$, $\vs_n =  \frac{1}{2}(\vDelta_{n+1} - \vDelta_{n})$, and $\ell_n(\mB) = \|\vy_n - \mB \vs_n\|^2$
        \State Update $\mB_{n+1}$ using Subroutine~\ref{alg:hessian_approx}\quad\Comment{See Section~\ref{subsec:hessian_approx_update}}
    \EndFor
    \State Set $\vw_t^k = \vw_{(k-1)T + t}$ for $k = 1,\dots,K$ and $t = 1, \dots, T$
    \State Set $\bar{\vw}^k = \frac{1}{T} \sum_{t=1}^T \vw_t^k$ for $k = 1,\dots,K$
    
    \State \Return $\hat{\vw} = \argmin_{\bar{\vw} \in \{\bar{\vw}^1, \ldots, \bar{\vw}^K\}} \|\nabla f(\bar{\vw})\|$
    \end{algorithmic}
    \end{algorithm}
For the sake of argument, assume we have access to the function's second-order information. In that case, a natural choice for the hint  \( \vh_n \) would be $
\vh_n = \nabla f(\vz_{n-1}) + \nabla^2 f(\vz_{n-1})(\vw_n - \vz_{n-1})$.
This expression offers a more accurate approximation of the gradient \( \vg_n \) compared to simply using 
the gradient \(\nabla f(\vz_{n-1})\)
as the hint, since it incorporates curvature information through the Hessian. Further, given the definition of $\vw_n$, we have $
\vw_n - \vz_{n-1} = \frac{\vDelta_n - \vDelta_{n-1}}{2}$. Substituting this into the hint vector, we obtain $
\vh_n = \nabla f(\vz_{n-1}) +\frac{1}{2}\nabla^2 f(\vz_{n-1})(\vDelta_n - \vDelta_{n-1})$.
While this hint is more accurate, it introduces two challenges: (i) The Hessian \( \nabla^2 f(\vz_{n-1}) \) is not available in our setting.
(ii) The hint depends on \( \vDelta_n \), making the update implicit, since \( \vDelta_n \) appears on both sides of the update equation.
To overcome the first issue, we replace the Hessian with an approximate matrix $\mB_n$ that depends only on gradient information, as is common in quasi-Newton methods. 
To address the second issue, we develop an efficient subroutine that allows us to perform the resulting implicit update effectively. This ensures that the dependency on \( \vw_n \) does not impede the computational efficiency of the algorithm.
To summarize, the hint function that we propose is:
\begin{equation}\label{hint}
\vh_{n+1} = {\nabla f(\vz_{n}) + \frac{1}{2}\mB_n (\vDelta_{n+1} - \vDelta_{n})} \quad \forall n \geq 1,    
\end{equation}
and for the initial step we set $ \vh_1 = \nabla f(\vx_0)$. We will later clarify how the matrix $\mB_n$ is selected.  Given this hint function, the update for our quasi-Newton optimistic method to pick the next action for {Online Learning Problem~\ref{box:OL1}} is given by:
\begin{equation}\label{eq:implicit_quasi_newton}
    {\vDelta_{n+1} \!=\! \Pi_{\|\vDelta\|\leq D} \Bigl[\vDelta_n - \eta \Bigl( \nabla f(\vz_{n}) + \frac{\mB_n}{2} (\vDelta_{n+1} \!-\! \vDelta_{n})  \Bigr)  - \eta \Bigl(\vg_n - \nabla f(\vz_{n-1}) - \frac{\mB_{n-1}}{2} (\vDelta_{n} \!-\! \vDelta_{n-1})\Bigr)\Bigr].}
\end{equation}
{Our proposed method is summarized in Algorithm~\ref{alg:conversion}.}
Now, there are two key questions we need to address: how to efficiently perform the update in \eqref{eq:implicit_quasi_newton}, given that it is an implicit update where \( \vDelta_{n+1} \) appears on both sides of the equation, and how to select the matrix \( \mB_n \). We subsequently address these questions in the following sections.

\subsection{Efficient Subroutine for OQN Update}\label{subsec:TRSolver}

As noted earlier, the update rule in \eqref{eq:implicit_quasi_newton} is implicit since the right-hand side also depends on \(\vDelta_{n+1}\). However, this update can still be efficiently executed by casting it as a solution to an inclusion problem, which closely resembles a trust-region problem. To highlight this connection, we first introduce a sequence of matrices and a sequence of vectors:
\begin{equation}\label{eq:def_An_bn}
    \mA_n = \frac{1}{2}\mB_n + \frac{1}{\eta} \mI \quad \text{and} \quad \vb_n = \nabla f(\vz_n)+\vg_n - \vh_n - \frac{1}{2}\mB_n \vDelta_n - \frac{1}{\eta}\vDelta_n.
\end{equation}
The following lemma shows that the update in \eqref{eq:implicit_quasi_newton} relates to an inclusion problem with $\mA_n$ and $\vb_n$.

\begin{lemma}\label{lem:inclusion}
Implementing the update in~\eqref{eq:implicit_quasi_newton} is equivalent to solving the following inclusion problem  $
    0 \in \mA_n\vDelta_{n+1} + \vb_n + \mathcal{N}_{\{\|\vDelta\| \leq D\}}({\vDelta_{n+1}})$,
where $\mathcal{N}_{\{\|\vDelta\| \leq D\}}(\vDelta_{n+1})$ denotes the normal cone to the set $\{\|\vDelta\| \leq D\}$ at the point $\vDelta_{n+1}$.  
\end{lemma}

With this connection, it becomes straightforward to relate \eqref{eq:implicit_quasi_newton} to the subproblem in trust-region methods. Specifically, note that the inclusion problem above corresponds to the first-order optimality condition for the following trust-region problem~\cite{conn2000trust}:
    \begin{equation}\label{eq:trust_region}
        \min_{\|\vDelta\| \leq D} \left\{ \frac{1}{2}\vDelta^\top \mA_n\vDelta  + \langle  \vb_n,  \vDelta \rangle \right\}.
    \end{equation}
    Thus, by finding a solution \(\vDelta_{n+1}\) that satisfies the first-order optimality condition for \eqref{eq:trust_region}, we also obtain a \(\vDelta_{n+1}\) that satisfies \eqref{eq:implicit_quasi_newton}.
Furthermore, as we will demonstrate, it suffices to solve the subproblem in \eqref{eq:trust_region} to a specified accuracy~\(\delta\). Specifically, we set  $
    \vDelta_{n+1} = \mathsf{TRSolver}(\mA_n, \vb_n, D, \delta)$,
where the \(\mathsf{TRSolver}\) oracle is defined as follows.

\begin{definition}\label{def:TRSolver}
    Given the inputs $\mA \in \mathbb{S}^d$, $\vb \in \reals^d$, $D>0$ and $\delta>0$, the $\mathsf{TRSolver}(\mA,\vb,D,\delta)$ oracle returns $\hat{\vDelta} \in \mathbb{R}^d$ such that $\|\hat{\vDelta}\| \leq D$ and there exists $\vv \in \mathcal{N}_{\{\|\vDelta\|\leq D\}}(\hat{\vDelta})$ with $\|\mA \hat{\vDelta} + \vb + \vv \| \leq \delta$.  
\end{definition}

In Lemma~\ref{lem:learning_regret}, we analyze how errors from solving the trust-region subproblem affect the update in \eqref{eq:implicit_quasi_newton} and the regret analysis. We then select an accuracy level to optimize overall complexity.

\subsection{Hessian approximation update via online learning} \label{subsec:hessian_approx_update}

In this section, we focus on the selection and update of the matrix \( \mB_n \) used in the hint function of our optimistic quasi-Newton method. To design its update, we first analyze the regret defined in Proposition~\ref{pr:average_gradient}, which the update in \eqref{eq:implicit_quasi_newton} aims to minimize. We then demonstrate that the best regret guarantee can be achieved if the Hessian approximation matrices  \( \mB_n \) follow an additional online learning update in the matrix space. In other words, we propose an online learning scheme for updating \( \mB_n \), motivated by the regret analysis of the optimistic method in \eqref{eq:implicit_quasi_newton} for solving the main online learning problem of finding a stationary point of \( f \).
Next, we characterize the regret obtained by performing the update of the optimistic quasi-Newton method proposed in \eqref{eq:implicit_quasi_newton}. 

\begin{lemma} \label{lem:learning_regret}
    Consider the optimistic quasi-Newton update in \eqref{eq:implicit_quasi_newton} for solving the online learning problem described above. If we define 
    $\vy_n:= {\vg_{n+1}-\nabla f(\vz_n)}$, $\vs_n := {\vw_{n+1} - \vz_n = \frac{1}{2}(\vDelta_{n+1} - \vDelta_{n})}$, and $\delta$ is the accuracy level for solving~\eqref{eq:trust_region}, then we have:
    $$
    \Reg_T(\vu^1,\dots,\vu^K) \leq \frac{4KD^2}{\eta} + \frac{3\eta}{2} \sum_{n=1}^{KT} \| \vy_n -  \mB_n \vs_n \|^2 + {2DKT \delta}.
    $$
\end{lemma}

This lemma shows how the regret in  OQN depends on the choice of matrices \( \mB_n \). To achieve the smallest regret bound, we need to minimize \(\sum_{n=1}^{KT} \| \vy_n - \mB_n \vs_n \|^2 \). The challenge is that both \( \vy_n \) and \( \vs_n \) depend on \( \mB_n \): we select \( \mB_n \), compute \( \vDelta_{n+1} \), then determine \( \vy_n \) and \( \vs_n \). 
Hence,  minimizing the cumulative sum associated with the choice of \( \mB_n \) can itself be formulated as an online learning problem. 
Given that Assumption~\ref{asm:L1} implies \( -L_1 \mI \preceq \nabla^2 f(x) \preceq L_1 \mI \), it is reasonable to select matrices from the set \( \mathcal{Z} \triangleq \{\mB \in \sS^d: \|\mB\|_{\op} \leq L_1\} \). { Specifically, this constraint not only aligns with Assumption~\ref{asm:L1} but also allows us to characterize the computational cost of \(\mathsf{TRSolver}\) and to bound the dynamic regret for our Hessian approximation online learning problem, as explained in the following remarks.}

\begin{remark}
    If we remove the constraint and consider an unconstrained online learning problem, two key issues arise. First, as discussed in Section~\ref{subsec:TRSolver}, our algorithm involves solving a trust-region subproblem~\eqref{eq:trust_region} every iteration, which depends on $\mathbf{B}_n$. As we shall establish in Lemma~\ref{lem:TRSolver}, the number of matrix-vector products required by $\mathsf{TRSolver}$  scales with $\sqrt{\|\mB_n\|_{\op}}$. Therefore, without a bound on $\mB_n$, the total computational cost of our algorithm cannot be controlled. Second, in the absence of constraints, it becomes more challenging to bound the dynamic regret in terms of the path length, as done in Lemma~\ref{lem:Dyn_regret_bound}. However, this issue may be addressed using more advanced techniques, as proposed in recent works~\cite{jacobsen2022parameter,zhang2022pde,luo2022corralling, jacobsen2023unconstrained,jacobsen2024an}.
\end{remark}
\begin{remark}
Instead of constraining the operator norm, one could alternatively impose a constraint on the Frobenius norm, selecting the matrices from the set $\{\mathbf{\mB} \in \sS^d: \|\mB\|_{F} \leq L_1\sqrt{d}\}$. 
In this case, we can establish a regret bound similar to Lemma~\ref{lem:Dyn_regret_bound}. However, this approach does not resolve the computational cost issue. To bound the cost of $\mathsf{TRSolver}$ using Lemma~\ref{lem:TRSolver}, we would need to bound the operator norm by $ \|\mB_n\|_{\op} \leq \|\mB_n\|_{F} \leq L_1\sqrt{d}$,  
    which introduces a worse dependence on the problem’s dimension~$d$.
\end{remark}

\refstepcounter{boxcounter}
\begin{mdframed}[linewidth=1pt]
    \textbf{Online Learning Problem 2 %
    } \\[5pt]
    For $n=1,\dots,KT$: 
    \begin{itemize}
        \item The learner chooses $\mB_n \in \mathcal{Z}$ 
        \item It observes the quadratic loss function \( \ell_n(\mB_n)=\| \vy_n -  \mB_n \vs_n \|^2 \)
    \end{itemize} 
    \textbf{Goal:} Minimize the cumulative loss given by $\sum_{n=1}^{KT} \|\vy_n - \mB_n\vs_n\|^2$.
\end{mdframed}
To update \(\mB_n\), one could simply apply the projected online gradient descent  (POGD) update with stepsize $\gamma$, which is given by $
\mB_{n+1} = \Pi_{\mathcal{Z}}\left( \mB_n - \gamma \nabla \ell_n(\mB_n) \right)$.
However, this approach requires a costly projection onto the set \(\mathcal{Z}\), involving a full eigenvalue decomposition with a cost of \(\mathcal{O}(d^3)\). To avoid this, we adopt the projection-free online learning framework proposed in {\cite{mhammedi2022efficient} and later developed in {\cite{jiang2023online,jiang2023accelerated}}}, which bypasses the need for full projection. {Instead, it builds on an approximate separation oracle for the feasible set $\mathcal{Z} = \{\mB\in \sS^d: \|\mB\|_{\op} \leq L_1\}$ defined below. As we later discuss in Section~\ref{subsec:computational},} it only requires calculating the largest and smallest eigenvalues, maintaining a cost of \(\mathcal{O}(d^2)\). Notably, the cumulative loss for \(\ell_n(\mB_n)\) in this projection-free method matches the cumulative loss achieved by POGD.

\begin{subroutine}[!t]\small
\caption{Online Learning Guided Hessian Approximation Update}\label{alg:hessian_approx}
  \begin{algorithmic}[1]
      \State \textbf{Input:} Initial matrix $\mB_1\in \mathbb{S}^d$ s.t. $\|\mB_1\|_{\op} \leq L_1$, step size $\rho>0$ %
      \State \textbf{Initialize:} set $\mW_1 \leftarrow \mB_1$
      and $\tilde{\mG}_1 \leftarrow \nabla \ell_1(\mB_1)$
      \State {Update $\mW_{2} \leftarrow \frac{\sqrt{d}L_1}{\max\{\sqrt{d}L_1,\|\mW_0 - \rho \tilde{\mG}_0\|_F\}}(\mW_0 - \rho \tilde{\mG}_0)$}
      \For{$n=1,\dots,M-1$}
      \State Query the oracle $(\gamma_n,\mS_n) \leftarrow \mathsf{SEP}(\mW_n)$ \label{line:SEPcall}
      \If{$\gamma_n \leq 1$}
        \State Set ${\mB}_n \leftarrow \mW_n$
        and $\tilde{\mG}_n \leftarrow \nabla \ell_n(\mB_n)$
      \Else
        \State Set ${\mB}_n \leftarrow \mW_n/\gamma_n$
        and $\tilde{\mG}_n \leftarrow \nabla \ell_n(\mB_n)+\max\{0,-\langle \nabla \ell_n(\mB_n), \mB_n \rangle\} \mS_n$
      \EndIf
      \State \label{line:projection}%
        Update 
        $\mW_{n+1} \leftarrow \frac{\sqrt{d}L_1}{\max\{\sqrt{d} L_1,\|\mW_n - \rho \tilde{\mG}_n\|_F\}}(\mW_n - \rho \tilde{\mG}_n)$ \quad\Comment{Euclidean projection onto $\mathcal{B}_{\sqrt{d}L_1}(0)$} 
      \EndFor
  \end{algorithmic}
\end{subroutine}
\begin{definition}\label{def:extevec}
    The $\mathsf{SEP}(\mW)$ oracle takes $\mW \in \mathbb{S}^d$ as input and returns a scalar $\gamma>0$ and a matrix $\mS\in \mathbb{S}^d$ with one of the following possible outcomes:
    \begin{itemize}
      \item \textbf{Case I:} $\gamma \leq 1$, which implies that $\|\mW\|_{\op} \leq 2L_1$.
      \item \textbf{Case II:} $\gamma>1$, which implies that $\|\mW/\gamma\|_{\op} \leq 2L_1$, $\|\mS\|_F \leq 1/L_1$ and $\langle \mS,\mW-{\mB}\rangle \geq \gamma -1$ for any ${\mB} \in \sS^d$ such that $\|{\mB}\|_{\op} \leq L_1$.
    \end{itemize}
  \end{definition}

  In words, %
  there are two possible outcomes for a given input~$\mW\in \sS^d$: we either certify that $\mW \in 2 \mathcal{Z}$, or we find a scaling factor $\gamma >1$ such that the scaled matrix $\mW/\gamma\in 2 \mathcal{Z}$ and a separating hyperplane given by $\mS$ between $\mW$ and the set $\mathcal{Z}$.

Equipped with the $\mathsf{SEP}$ oracle, we are ready to present our projection-free online learning algorithm for $\{\mB_n\}_{n\geq 0}$, which is given in Subroutine~\ref{alg:hessian_approx}. The core idea is to introduce an auxiliary online learning problem %
over a larger feasible set, defined as $\mathcal{B}_{\sqrt{d}L_1}(0) = \{\mW \in \sS^d: \|\mW\|_F \leq \sqrt{d}L_1\}$, where projections are easy to compute. This auxiliary problem employs carefully designed surrogate loss functions $\tilde{\ell}_n(\mW) = \langle \tilde{\mG}_n,\mW\rangle$, where $\tilde{\mG}_n \in \sS^d$ will be defined later.
Instead of directly tackling the original online learning problem, we apply projected online gradient descent on the surrogate loss function $\tilde{\ell}_n(\mW)$ to update the auxiliary iterates $\{\mW_n\}_{n\geq 0}$ (see Line~\ref{line:projection} in Subroutine~\ref{alg:hessian_approx}). The sequence $\{\mB_n\}_{n\geq 0}$ is then generated from $\{\mW_n\}_{n \geq 0}$ using the $\mathsf{SEP}$ oracle. Specifically, let $\gamma_n$ and $\mS_n$ denote the output of $\mathsf{SEP}(\mW_n)$. 
If $\gamma_n \leq 1$, this certifies that $\|\mW_n\|_{\op} \leq 2L_1$ and we set $\mB_n \leftarrow \mW_n$ and $\tilde{\mG}_n \leftarrow \nabla \ell_n(\mB_n)$. Otherwise, if $\gamma_n > 1$, we rescale~$\mW_n$ to obtain $\mB_n \leftarrow \mW_n/\gamma_n$ and update $\tilde{\mG}_n \leftarrow \nabla \ell_n(\mB_n)+\max\{0,-\langle \nabla \ell_n(\mB_n), \mB_n \rangle\} \mS_n$. By Definition~\ref{def:extevec}, this ensures that $\|\mB_n\|_{\op}\leq 2L_1$. To demystify our choice of $\tilde{\mG}_n$, note that the surrogate loss function is designed such that the immediate regret of the auxiliary online learning problem serves an upper bound for the original problem, i.e., we have ${\ell}_{n}(\mB_n) - \ell_n(\mB) \leq \tilde{\ell}_n(\mW) - \tilde{\ell}_n(\mB)$ for any $\mB \in \mathcal{Z}$. Thus, this allows us to apply the standard regret analysis of projected OGD to bound the regret of the auxiliary problem, which in turn implies a regret bound for the original problem. 
\begin{remark}
 Unlike the standard online learning setting, we do not strictly require $\mB \in \mathcal{Z}$, but we ensure $\|\mB\|_{\op} \leq 2L_1$, equivalent to $\mB \in 2\mathcal{Z}$. This relaxation suffices for our analysis.
\end{remark}
\section{Complexity Analysis}

\looseness=-1
To characterize the overall complexity of the proposed method, we begin by establishing the regret associated with the updates in our proposed optimistic method. Building on Lemma \ref{lem:learning_regret} and Proposition \ref{pr:average_gradient}, we study the \textit{gradient complexity} of our method. Next, we analyze the computational cost of the $\mathsf{TRSolver}$ oracle used in Algorithm~\ref{alg:conversion}, as well as the computational cost of the projection-free online learning scheme for updating the Hessian approximation in Subroutine~\ref{alg:hessian_approx}. These analyses together allow us to characterize the total \textit{computational cost} of our proposed method.

\subsection{Convergence Rate}

In this section, we present the final convergence rate of our algorithm by selecting appropriate values for the hyperparameters. Notably, we have four free parameters, \(\eta\), \(D\), \(K\), and \(T\), that can be chosen subject to the constraint \(KT = M\). We also characterize the accuracy level \(\delta\) for solving the subproblem to ensure that the overall complexity remains unaffected by this inexactness.

\begin{theorem}\label{thm:convergence_rate}
Suppose Assumptions \ref{asm:L1} and \ref{asm:L2} hold. If we run  Algorithm~\ref{alg:conversion} with parameters  $ D = \Theta ( \bigl( \frac{f(\vx_0) - f^*}{d^{{2}/{5}} L_1^{{2}/{5}} L_2^{{3}/{5}} M} \bigr)^{\frac{5}{13}} ),  \eta = \Theta ( ( \frac{1} {dL_1L_2^{{2}/{3}}D^{{2}/{3}}} )^{\frac{3}{5}} ) , T =  \Theta ( {(DL_2\eta)^{-\frac{1}{3}}} )  $, and {$\delta = \frac{D}{\eta T}$}, then we have:
    \begin{equation}   
            \frac{1}{K}\sum_{k=1}^K \|\nabla f(\bar{\vw}^k)\| = \bigO \left(
            (f(\vx_0) - f^*)^{\frac{8}{13}} L_1^{\frac{2}{13}} L_2^{\frac{3}{13}} \frac{d^{\frac{2}{13}}}{M^{\frac{8}{13}}} \right). 
    \end{equation}
\end{theorem}

This result implies that after \(M\) iterations, \(\min_{1 \leq k \leq K} \|\nabla f(\bar{\vw}^k)\|\) is at most \(\mathcal{O}(\nicefrac{d^{\frac{2}{13}}}{M^{\frac{8}{13}}})\). Thus, to ensure we find an \(\varepsilon\)-FOSP, we require at most \(M = \mathcal{O}(d^{1/4}/\varepsilon^{13/8})\) iterations, giving our method a gradient complexity of \(\mathcal{O}(d^{1/4}/\varepsilon^{13/8})\). As discussed, the best-known complexity bound for achieving an \(\varepsilon\)-FOSP with a first-order oracle under Assumptions \ref{asm:L1} and \ref{asm:L2} is \(\mathcal{O}(1/\varepsilon^{7/4})\) gradient queries. The above result shows that our method improves this complexity when \(d = \mathcal{O}(1/\sqrt{\varepsilon})\).

\begin{proof}[Proof Sketch]
Based on Proposition \ref{pr:average_gradient}, achieving the final convergence rate requires bounding the \(K\)-shifting regret. By Lemma \ref{lem:learning_regret}, this entails controlling the loss \(\ell_n (\mB_n) = \|\vy_n - \mB_n \vs_n\|^2\). To address this, we use a dynamic regret analysis, where our algorithm sequentially selects actions while competing against an adversary with an alternative action sequence. The dynamic regret is formally defined as: $D\text{-}\Reg(\mH_1, \dots, \mH_{KT}) = \sum_{n=1}^{KT} \ell_n(\mB_n) - \ell_n(\mH_n)$, 
 and our goal now is to minimize $D\text{-}\Reg(\mH_1, \dots, \mH_{KT})$. We demonstrate that our projection-free online learning method for updating the Hessian approximation yields the following result.

\begin{lemma}\label{lem:Dyn_regret_bound}
    Let $\rho = \frac{1}{16D^2}$ in Algorithm~\ref{alg:hessian_approx}. Then we have:
    \begin{equation*}
        \sum_{n=1}^{KT} \ell_n(\mB_n) \leq 16D^2 \|\mW_{1} - \mH_1\|_F^2 +  2 \sum_{n=1}^{KT} \ell_n(\mH_n)   + 64 L_1 D^2 \sqrt{d} \sum_{n=1}^{KT}  \| \mH_{n+1}  - \mH_{n} \|_F.
    \end{equation*}
\end{lemma}
Next, we establish upper bounds for the terms on the right-hand side of the above lemma. First, note that \(\|\mW_{1} - \mH_1\|_F^2 \leq 4dL_1^2\). To bound \(\sum_{n=1}^{KT} \ell_n(\mH_n)\) and the path length \(\sum_{n=1}^{KT} \| \mH_{n+1} - \mH_{n} \|_F\), we select an appropriate competitor sequence \(\{\mH_i\}_{i=1}^n\) by setting \(\mH_n = \nabla^2 f(\vz_{n})\). Using this choice, we apply the following result to bound the last two terms in Lemma~\ref{lem:Dyn_regret_bound}.
\begin{lemma}\label{lem:H_n_loss}
   Recall that $\vz_n = \vx_n+ \frac{1}{2}\vDelta_n$. If we set $\mH_n = \nabla^2 f(\vz_{n})$, then we have $\ell_n( \mH_n ) \leq \frac{L_2^2}{4}D^4$ and $\| \mH_{n+1}  - \mH_n \|_F \leq 2 L_2 \sqrt{d} D$. 
\end{lemma}
Combining Lemmas \ref{lem:Dyn_regret_bound} and \ref{lem:H_n_loss} implies 
$\sum_{n =1}^{KT} \ell_n(\mB_n) \leq \frac{L_2^2D^4 KT}{2} + 64dL_1^2D^2 + 128dL_1L_2D^3KT$.
Then, given Lemma \ref{lem:learning_regret}, if we set $\delta = \frac{D}{\eta T}$, then we obtain:
\begin{align*}
    \Reg_T(\vu^1,\dots,\vu^K)  \leq \frac{6KD^2}{\eta} + \frac{3\eta}{2} \left(  \frac{KTL_2^2D^4}{2} + 64dL_1^2D^2 + 128dL_1L_2D^3KT  \right),
\end{align*} 
and the finally by leveraging Proposition~\ref{pr:average_gradient}, we obtain:
\begin{equation*}
        \frac{1}{K}\sum_{k=1}^K \|\nabla f(\bar{\vw}^k)\| \leq \frac{f(\vx_0) - f^*}{DM}
        +\frac{6D}{T\eta} + \frac{3\eta}{2} \left[  \frac{L_2^2 D^3}{2} + \frac{64 dL_1^2D}{M} + 128 d L_1 L_2 D^2  \right]
        + \frac{L_2 D^2}{48} + \frac{L_2D^2T^2 }{2}.
\end{equation*}
Finally, by optimizing the free parameters \(\eta\), \(D\), \(K\), and \(T\) as suggested in the theorem's statement, the main claim follows.
\end{proof}

\subsection{Characterizing the Computational Cost}
\label{subsec:computational}

\myalert{Implementation of the $\mathsf{TRSolver}$ oracle.}
We first discuss the implementation of  $\mathsf{TRSolver}$ used in our optimistic quasi-Newton algorithm (see Algorithm~\ref{alg:conversion}).  Recall the trust-region problem: 
\begin{equation}\label{eq:trust_region_generic}
    \min_{\|\vDelta\| \leq D} \left\{ \frac{1}{2}\vDelta^\top \mA\vDelta  + \langle  \vb,  \vDelta \rangle \right\}.
\end{equation}
We aim to find $\hat{\vDelta}$ such that $\|\mA \hat{\vDelta} + \vb + \vv\| \leq \delta$ for some $\vv$ in the normal cone $\mathcal{N}_{\{\|\vDelta\| \leq D\}}(\hat{\vDelta})$. 
Inspired by \cite{honguyen2017second,wang2017linear}, we reformulate \eqref{eq:trust_region_generic} as a convex minimization problem. First, we approximately compute the minimum eigenvalue of $\mA$, which we denote by $\lambda_{\min}(\mA)$. If $\lambda_{\min}(\mA) \geq 0$, then the problem in \eqref{eq:trust_region_generic} is already convex, and there exist fast algorithms to find an approximate first-order stationary point. Specifically, we will apply  $\mathsf{FISTA\mathrm{+}SFG}$ proposed in \cite{lee2021geometric,kim2023time}, and it is shown that we can find $\hat{\vDelta}$ satisfying the condition after $\bigO\left(\sqrt{\nicefrac{\lambda_{\max}(\mA) D}{\delta}}\right)$ iterations (see Appendix~\ref{appen:gradient_norm}). Otherwise, if $\lambda_{\min}(\mA) <0$, we can instead consider a regularized problem:
\begin{equation}\label{eq:trust_region_regularized}
    \min_{\|\vDelta\| \leq D} \left\{ \frac{1}{2}\vDelta^\top (\mA - \lambda_{\min}(\mA) \mI)\vDelta  + \langle  \vb,  \vDelta \rangle \right\},
\end{equation}
which is convex. Hence, we can apply $\mathsf{FISTA\mathrm{+}SFG}$ to obtain an approximate first-order stationary point $\tilde{\vDelta}$ of \eqref{eq:trust_region_regularized}. Moreover, if 
$\vv_{\min}$ is an (approximate) eigenvector corresponding to the minimum eigenvalue of $\mA$, then we can construct the solution $\hat{\vDelta}$ to \eqref{eq:trust_region_generic} from a linear combination of $\tilde{\vDelta}$ and~$\vv_{\min}$. 
We defer the details to Appendix~\ref{appen:trsolver} and present the following lemma. 
\begin{lemma}\label{lem:TRSolver}
    Given inputs $\mA \in \mathbb{S}^{d}$, $\vb \in \reals^d$, $D >0$ and $\delta>0$, 
    let $B$ be an upper bound on $\max\{\lambda_{\max}(\mA) -\lambda_{\min}(\mA), \lambda_{\max}(\mA)\}$.
    Then,  $\mathsf{TRSolver}(\mA, \vb, D, \delta)$ requires  $\tilde{\bigO}\Bigl(\sqrt{\frac{BD}{\delta}}\Bigr)$ matrix-vector products with success probability at least $1\!-\!q$, where we hide logarithmic terms of $d$, $\!B\!$, $\!D$\!, $\!q$, and $\delta$.  
\end{lemma}

\myalert{Implementation of the $\mathsf{SEP}$ oracle.} 
In this part, we provide implementation details for the $\mathsf{SEP}$ oracle used in our projection-free online learning algorithm (see Subroutine~\ref{alg:hessian_approx}). While this oracle has also been employed in \cite{jiang2023online}, we include its construction here for completeness. 

As noted in~\cite{jiang2023online}, the $\mathsf{SEP}$ oracle is closely related to computing the extreme eigenvalues and eigenvectors of a given matrix. %
Specifically, for an input matrix $\mW \in \sS^d$, let $\lambda_{\max}(\mW)$ and $\lambda_{\min}(\mW)$ denote the maximum and minimum eigenvalues of $\mW$, with corresponding unit eigenvectors $\vv_{\max}$ and $\vv_{\min}$, respectively. Note that $\|\mW\|_{\op} = \max\{\lambda_{\max}(\mW), -\lambda_{\min}(\mW)\}$. Hence, by setting $\gamma = \frac{\max\{\lambda_{\max}(\mW), -\lambda_{\min}(\mW)\}}{L_1}$, we observe that $\gamma \leq 1$ certifies $\|\mW\|_{\op} \leq L_1$ (\textbf{Case I} in Definition~\ref{def:extevec}). Otherwise, if $\gamma >1$ (\textbf{Case II}), then scaling $\mW$ by $\gamma$ yields $|\mW/\gamma|{\op} = L_1$, and the eigenvectors $\vv_{\max}$ and $\vv_{\min}$ can be used to construct the separating hyperplane. Indeed, assume without loss of generality that $\lambda_{\max}(\mW) \geq - \lambda_{\min}(\mW)$. By setting $\mS =  \frac{1}{L_1}\vv_{\max} \vv_{\max}^\top$, this ensures that $\|\mS\|_F \leq 1/L_1$, and for any $\mB$ satisfying $\|\mB\|_{\op} \leq L_1$, it holds that  $\langle \mS, \mW - \mB\rangle = \frac{1}{L_1}(\vv_{\max}^\top\mW\vv_{\max} -\vv_{\max}^\top\mB\vv_{\max}) \leq \frac{1}{L_1}(\lambda_{\max}(\mW) -L_1) = \gamma-1$. 
Moreover, as discussed in Appendix~\ref{appen:SEP}, it is sufficient to compute the extreme eigenvalues and eigenvectors of $\mW$ approximately, which can be achieved efficiently by the Lanczos algorithm with a random start~\cite{kuczynski1992estimating}. We summarize the computation cost in the following lemma and defer the details to Appendix~\ref{appen:SEP}. 

\begin{lemma}\label{lem:SEP}
    Given an input $\mW \in \sS^d$, we can implement $\mathsf{SEP}(\mW)$ using $\bigO\left(\log \frac{d}{q^2}\right)$ matrix-vector products with success probability at least $1-q$.
\end{lemma}

Now we are ready to state the total computational cost of Algorithm~\ref{alg:conversion} in terms of the number of matrix-vector products. Note that each iteration of Algorithm~\ref{alg:conversion} requires one call to the $\mathsf{TRSolver}$ oracle and one call to the $\mathsf{SEP}$ oracle. By using Lemmas~\ref{lem:TRSolver} and~\ref{lem:SEP} and setting the proper parameters, we arrive at the following result. The proof is given in Appendix~\ref{appen:computational}.

\begin{theorem}\label{thm:computational}
\looseness = -1
    To find an $\epsilon$-FOSP using Algorithm~\ref{alg:conversion}, the total number of matrix-vector products required by the $\mathsf{TRSolver}$ and $\mathsf{SEP}$ oracles are bounded by $\tilde{\bigO}\Bigl( \frac{d^{1/8}}{\epsilon^{29/16}} + \frac{d^{3/8}}{\epsilon^{27/16}} \Bigr) $
    and $\tilde{\bigO}\Bigl( \frac{d^{1/4}}{\epsilon^{13/8}} \Bigr)$, respectively. 
\end{theorem}

\section{Discussion on the Lower Bound and Concluding Remarks}

In this paper, we proposed a novel method for improving the complexity of finding an \(\varepsilon\)-first-order stationary point (FOSP) by integrating a two-level online learning framework. Our results demonstrated that by leveraging this method, we can achieve an improved complexity bound of \(\mathcal{O}( \frac{d^{1/4}}{\varepsilon^{13/8}})\), which surpasses the best-known bound with a first-order oracle if \(d = \mathcal{O}( \frac{1}{\sqrt{\varepsilon}} )\).

To better understand the upper bound on the complexity of our proposed method, we compare it with the best-known lower bound for the relevant setting. Specifically, for the setting where the gradient and Hessian are both Lipschitz and only gradient information is accessible, the authors in  \cite{carmon2021lower} established  a lower bound of \(\Omega( \frac{1}{\varepsilon^{12/7}} )\) for the number of gradient queries required to find an \(\varepsilon\)-FOSP. While our dependency on \(\varepsilon\) initially appears more favorable than this lower bound, our result does not violate it because the lower bound only applies when \(d \geq \frac{1}{\varepsilon^{12/7}}\). Since our bound also depends on \(d\), when \(d \geq \frac{1}{\varepsilon^{12/7}}\), our upper bound could degrade to \(\mathcal{O}( \frac{1}{\varepsilon^{(3/7)+(13/8)}} )\), which simplifies to \(\mathcal{O}( \frac{1}{\varepsilon^{115/56}} )\). This complexity is indeed worse than the lower bound of \(\Omega( \frac{1}{\varepsilon^{12/7}})\).

Based on the above discussion, a promising direction for future research is to establish a lower bound that is either dimension-independent, applying for any \(d\), or specifically suited to small-to-moderate values of \(d\), particularly in the regime \(d \leq \frac{1}{\sqrt{\varepsilon}}\), where our proposed method outperforms the known complexity of \(\mathcal{O}( \frac{1}{\varepsilon^{7/4}} )\).

%% file: appendix.tex
\section*{Appendix}
\section{
Proofs for Online-to-non-convex Conversion
}\label{Proof_appendix}

\subsection{Proof of Lemma \texorpdfstring{\ref{lem:conversion}}{2.1}}\label{appen:conversion_1_s}

By the fundamental theorem of calculus and using $\vx_{n} - \vx_{n-1}  =\vDelta_n$, we have 
\begin{equation}\label{eq:integral}
    f(\vx_n) - f(\vx_{n-1}) = \int_{0}^1 \langle \nabla f(\vx_{n-1}+s(\vx_n-\vx_{n-1})), \vx_n-\vx_{n-1}\rangle \; ds = \Bigl\langle \int_{0}^1 \nabla f(\vx_{n-1}+s\vDelta_n) \;ds,
 \vDelta_n\Bigr\rangle.
\end{equation}
Moreover, recall that $\vg_n = \nabla f(\vw_n)$ and $\vw_n = \frac{1}{2}(\vx_{n-1}+\vx_n) = \vx_{n-1}+\frac{1}{2}\vDelta_n$. Hence, we can bound 
\begin{align}
   &\phantom{{}={}}\left\| \int_{0}^1 \nabla f(\vx_{n-1}+s\vDelta_n) \;ds - \vg_n\right\| \nonumber\\
   &=  \left\| \int_{0}^1 \left(\nabla f(\vx_{n-1}+s\vDelta_n) -\nabla f\left(\vx_{n-1} + \frac{1}{2}\vDelta_n\right)\right) \;ds\right\| \nonumber\\
   & = \left\| \int_{0}^1 \left(\nabla f(\vx_{n-1}+s\vDelta_n) -\nabla f\left(\vx_{n-1} + \frac{1}{2}\vDelta_n\right) - \nabla^2 f\left(\vx_{n-1} + \frac{1}{2}\vDelta_n\right)\left(s-\frac{1}{2}\right)\vDelta_n\right) \;ds\right\| \label{eq:insert_hessian}\\
   &\leq  \int_{0}^1 \left\|\nabla f(\vx_{n-1}+s\vDelta_n) -\nabla f\left(\vx_{n-1} + \frac{1}{2}\vDelta_n\right) - \nabla^2 f\left(\vx_{n-1} + \frac{1}{2}\vDelta_n\right)\left(s-\frac{1}{2}\right)\vDelta_n\right\| \;ds \nonumber\\
   &\leq \int_{0}^1 \frac{L_2}{2}\|\vDelta_n\|^2 \left(s-\frac{1}{2}\right)^2 \;ds = \frac{1}{48}L_2 D^2,\label{eq:gradient_diff}
\end{align}
where we used the fact that $\int_{0}^1 (s-\frac{1}{2}) \; ds = 0$ in \eqref{eq:insert_hessian} and Assumption~\ref{asm:L2} in~\eqref{eq:gradient_diff}. Combining \eqref{eq:integral} and \eqref{eq:gradient_diff} leads to  
\begin{align*}
    f(\vx_n) - f(\vx_{n-1}) &= \left\langle \int_{0}^1 \nabla f(\vx_{n-1}+s\vDelta_n) \;ds,
 \vDelta_n\right\rangle \\
  &= \langle \vg_n, \vDelta_n \rangle + \left\langle \int_{0}^1 \nabla f(\vx_{n-1}+s\vDelta_n) \;ds -\vg_n, \vDelta_n\right\rangle \\
 &\leq \langle \vg_n, \vDelta_n \rangle + \left\| \int_{0}^1 \nabla f(\vx_{n-1}+s\vDelta_n) \;ds -\vg_n \right\| \|\vDelta_n\| \\
 & \leq \langle \vg_n, \vDelta_n \rangle + \frac{1}{48}L_2 D^3.
\end{align*}
where the first inequality follows from the Cauchy-Schwarz inequality and the second inequality is due to \eqref{eq:gradient_diff} and the fact that $\|\vDelta_n\| \leq D$. Lemma~\ref{lem:conversion} follows by rearranging the above inequality.   

\subsection{Proof of Lemma \texorpdfstring{\ref{lemma:averaged_gradient}}{2.2}}\label{averaged_gradient_s}
    Recall that $\bar{\vw} = \frac{1}{T} \sum_{n=1}^T \vw_n$, which implies that $\sum_{n=1}^T \nabla^2 f(\bar{\vw})(\vw_n 
 -\bar{\vw}) = 0$. Hence, we can write 
 \begin{align}
     \left\|\frac{1}{T}\sum_{n=1}^T \nabla f(\vw_n) - \nabla f(\bar{\vw})\right\| &= \left\|\frac{1}{T}\sum_{n=1}^T \left(\nabla f(\vw_n) - \nabla f(\bar{\vw})\right)\right\| \nonumber\\
     & = \left\|\frac{1}{T}\sum_{n=1}^T \left(\nabla f(\vw_n) - \nabla f(\bar{\vw}) - \nabla^2 f(\bar{\vw})(\vw_n 
 -\bar{\vw})\right)\right\| \nonumber\\
 & \leq \frac{1}{T}\sum_{n=1}^T \left\|\nabla f(\vw_n) - \nabla f(\bar{\vw}) - \nabla^2 f(\bar{\vw})(\vw_n 
 -\bar{\vw})\right\| \nonumber\\
 & \leq \frac{L_2}{2T} \sum_{n=1}^T \|\vw_n 
 -\bar{\vw}\|^2. \label{eq:gradient_difference_bound}
 \end{align}
 Here, the first inequality is due to the triangle inequality and the second inequality follows from Assumption~\ref{asm:L2}. 
 To bound $\|\vw_n - \bar{\vw}\|$, recall that $\vw_n = \vx_{n-1} + \frac{1}{2}\vDelta_n$ and $\vx_n = \vx_{n-1} + \vDelta_n$ for $1\leq n \leq T$. Consider any $s,t \in \{1,\dots,T\}$ and assume that $s > t$ without loss of generality. Then we have:
 \begin{align*}
     \vw_t- \vw_s =  \vx_{t-1}+\frac{1}{2}\vDelta_t - \vx_{s-1} - \frac{1}{2}\vDelta_s 
     &= \vx_{t-1} +\frac{1}{2}\vDelta_t - \Bigl(\vx_{t-1} + \sum_{i=t}^{s-1} \vDelta_i \Bigr) - \frac{1}{2}\vDelta_s \\
     &=  - \frac{1}{2}\vDelta_t - \sum_{i=t+1}^{s-1} \vDelta_i  - \frac{1}{2}\vDelta_s.  
 \end{align*}
 Since $\|\vDelta_n\| \leq D$ for any $n \in \{1,\dots,T\}$, it follows from the triangle inequality that $\|\vw_t- \vw_s\| \leq D(s-t) \leq TD$ for any $s,t \in \{1,\dots,T\}$. 
 Since $\bar{\vw} = \frac{1}{T} \sum_{n=1}^T \vw_n$,  we further have $\|\vw_n -\bar{\vw}\| \leq TD$ for any $n \in \{1,\dots,T\}$. Consequently, we obtain from \eqref{eq:gradient_difference_bound} that 
 \begin{equation*}
      \left\|\left(\frac{1}{T}\sum_{n=1}^T \nabla f(\vw_n) \right)- \nabla f(\bar{\vw})\right\| \leq \frac{L_2}{2}T^2 D^2. 
 \end{equation*}
 Therefore, it follows from the triangle inequality that
\begin{align*}
     \frac{L_2}{2}T^2 D^2 & \geq \left\|\frac{1}{T}\sum_{n=1}^T \nabla f(\vw_n) - \nabla f(\bar{\vw})\right\|  \geq \left\| \nabla f(\bar{\vw})\right\| - 
       \left\|\frac{1}{T}\sum_{n=1}^T \nabla f(\vw_n) \right\|. 
 \end{align*}
Noting that $\vg_n =\nabla f(\vw_n)$, this completes the proof of Lemma~\ref{lemma:averaged_gradient}.

\subsection{Proof of Proposition \texorpdfstring{\ref{pr:average_gradient}}{2.3}}
Consider the $k$-th episode ($k=1,2,\dots,K$) from $n = (k-1)T+1$ to $n = kT$. By applying the inequality in~\eqref{eq:bound_1}, we obtain 
\begin{equation*}
    \bigg\|\frac{1}{T}\sum_{n=(k-1)T+1}^{kT} \vg_n\bigg\|\leq  \frac{f(\vx_{(k-1)T})- f(\vx_{kT})}{DT} + \frac{1}{DT}\sum_{n=(k-1)T+1}^{kT} \vg_n^\top (\vDelta_n-\vu^k) +\frac{L_2D^2}{48}. 
\end{equation*} 
Moreover, recall that $\bar{\vw}^k = \frac{1}{T} \sum_{n=(k-1)T+1}^{kT} \vw_n$ and it follows from Lemma~\ref{lemma:averaged_gradient} that   
$\|\nabla f(\bar{\vw}^k)\| \leq \bigg\|\frac{1}{T}\sum_{n=(k-1)T+1}^{kT} \vg_n\bigg\| + \frac{L_2}{2}T^2 D^2$. Together with the above inequality, this leads to 
\begin{equation*}
    \|\nabla f(\bar{\vw}^k)\| \leq \frac{f(\vx_{(k-1)T})- f(\vx_{kT})}{DT} + \frac{1}{DT}\sum_{n=(k-1)T+1}^{kT} \vg_n^\top (\vDelta_n-\vu^k) +\frac{L_2D^2}{48} + \frac{L_2}{2}T^2 D^2. 
\end{equation*}
Summing the above inequality from $k=1$ to $k=K$ and dividing both sides by $K$ yields:  
\begin{equation*}
   \frac{1}{K} %
   \sum_{k=1}^K \|\nabla f(\bar{\vw}^k)\| 
   \leq  \frac{f(\vx_{0})- f(\vx_{M})}{KDT} + \frac{1}{KDT}\sum_{k=1}^K \sum_{n=(k-1)T+1}^{kT} \vg_n^\top (\vDelta_n-\vu^k) +\frac{L_2D^2}{48} + \frac{L_2}{2}T^2 D^2. 
\end{equation*}
Finally, we note that $f(\vx_M) \geq f^*$ and this completes the proof.

\section{Proofs for Section \texorpdfstring{\ref{sec:algorithm}}{3}}

\subsection{Proof of Lemma \texorpdfstring{\ref{lem:inclusion}}{3.1}} 
    Recall that $\vh_n = \nabla f(\vz_{n-1}) + \frac{1}{2}\mB_{n-1} (\vDelta_{n} - \vDelta_{n-1})$. 
    By the property of Euclidean projection, the equation in~\eqref{eq:implicit_quasi_newton} is satisfied if and only if 
    \begin{equation}\label{eq:inclusion}
        \vDelta_n - \eta \left( \nabla f(\vz_{n}) + \frac{1}{2}\mB_n (\vDelta_{n+1} - \vDelta_{n})  \right)  - \eta \left(\vg_n - \vh_n\right) - \vDelta_{n+1} \in \mathcal{N}_{\{\|\vDelta\| \leq D\}}({\vDelta_{n+1}}).
    \end{equation}
    Moreover, given our definitions in \eqref{eq:def_An_bn}, we can rewrite the right-hand side of \eqref{eq:inclusion} as $-\eta(\mA_n \vDelta_{n+1} + \vb_n)$. By dividing both sides by $\eta$ and rearranging, we obtain the inclusion problem stated in Lemma~\ref{lem:inclusion}.

\subsection{Proof of Lemma \texorpdfstring{\ref{lem:learning_regret}}{3.2}}

Before presenting the proof of Lemma~\ref{lem:learning_regret}, we first establish the following auxiliary lemma. 
\begin{lemma}\label{lem:one_step_regret}
    Consider 
    the update rule $
    \vDelta_{n+1} = \mathsf{TRSolver}(\mA_n, \vb_n, D, \delta)$.
    Then for any $\vu$ such that $\|\vu\| \leq D$, it holds that 
    \begin{align*}
        \langle \vg_{n+1}, \vDelta_{n+1} - \vu \rangle &\leq \frac{\|\vDelta_n-\vu\|^2}{2\eta}   - \frac{\|\vDelta_{n+1}-\vu\|^2}{2\eta} +  \langle \vg_{n+1} - \vh_{n+1}, \vDelta_{n+1} - \vu \rangle -  \langle \vg_{n} - \vh_{n}, \vDelta_{n} - \vu \rangle \\
        &\phantom{{}\leq{}} + \eta\|\vg_n - \vh_n\|^2 - \frac{1}{4\eta}\|\vDelta_{n+1} - \vDelta_n\|^2 + {2D \delta}.
    \end{align*}
\end{lemma}
\begin{proof}
{By using the definition of $\mathsf{TRSolver}$ in Definition~\ref{def:TRSolver}, there exists $\vv \in \mathcal{N}_{\{\|\vDelta\|\leq D\}}({\vDelta}_{n+1})$ such that $\|\mA_n \vDelta_{n+1} + \vb_n + \vv\| \leq \delta$.  Hence, by the Cauchy-Schwarz inequality, we obtain 
\begin{equation*}
    \langle \mA_n \vDelta_{n+1} + \vb_n + \vv, \vu - \vDelta_{n+1}  \rangle \geq - \|\mA_n \vDelta_{n+1} + \vb_n + \vv\| \|\vu - \vDelta_{n+1} \| = -2 D\delta. 
\end{equation*}
Moreover, it follows from the definition of the normal cone that $ \langle \vv, \vDelta_{n+1} - \vu\rangle \geq 0$ for any $\vu\in \{\vDelta \in \reals^d: \|\vDelta\|\leq D\}$. Hence, we further have 
\begin{equation}\label{eq:lower_bound_gap}
    \langle \mA_n \vDelta_{n+1} + \vb_n, \vu - \vDelta_{n+1}  \rangle \geq  -2 D\delta + \langle \vv, \vDelta_{n+1} - \vu\rangle \geq -2 D\delta. 
\end{equation}
From \eqref{eq:def_An_bn} and \eqref{hint}, we can rewrite $\mA_n \vDelta_{n+1} + \vb_n = \frac{1}{\eta}(\vDelta_{n+1}-\vDelta_n) + \vh_{n+1} + (\vg_n - \vh_n)$. Combining this with \eqref{eq:lower_bound_gap} leads to  
}
\begin{equation*}
    \bigl\langle \frac{1}{\eta}(\vDelta_{n+1}-\vDelta_n) +  \vh_{n+1} + (\vg_n - \vh_n), \vu - \vDelta_{n+1} \bigr\rangle \geq -2D \delta.
\end{equation*}
    Hence, it implies that 
    \begin{align*}
        \langle \vg_{n+1}, \vDelta_{n+1} - \vu \rangle &\leq \langle \vg_{n+1} - \vh_{n+1}, \vDelta_{n+1} - \vu \rangle - \langle \vg_{n} - \vh_{n}, \vDelta_{n+1} - \vu \rangle \\
        &\phantom{{}\leq{}}+ \frac{1}{2\eta}\|\vDelta_n-\vu\|^2 - \frac{1}{2\eta}\|\vDelta_{n+1} - \vDelta_n\|^2 - \frac{1}{2\eta}\|\vDelta_{n+1}-\vu\|^2 + 2D \delta,
    \end{align*}
    where we have used the three-point equality $ \langle \vDelta_{n+1}-\vDelta_n, \vu  - \vDelta_{n+1}  \rangle = \frac{1}{2}\|\vDelta_n-\vu\|^2 - \frac{1}{2}\|\vDelta_{n+1} - \vDelta_n\|^2 - \frac{1}{2}\|\vDelta_{n+1}-\vu\|^2$.
    Furthermore, we have:
    \begin{equation*}
        -\langle \vg_{n} - \vh_{n}, \vDelta_{n+1} - \vu \rangle = -\langle \vg_{n} - \vh_{n}, \vDelta_{n} - \vu \rangle + \langle \vg_{n} - \vh_{n}, \vDelta_n -\vDelta_{n+1} \rangle.
    \end{equation*}
We can further bound the second term as follows: 
\begin{equation*}
    \langle \vg_{n} - \vh_{n}, \vDelta_n -\vDelta_{n+1}  \rangle \leq \|\vg_{n} - \vh_{n}\| \|\vDelta_n -\vDelta_{n+1} \|     
    \leq \eta\|\vg_{n} - \vh_{n}\|^2 + \frac{1}{4 \eta}\|\vDelta_{n+1} - \vDelta_n\|^2,  
\end{equation*}
where the last inequality is due to a weighted version of Young's inequality. Combining all the inequalities gives us the desired result. 
\end{proof}

Using Lemma~\ref{lem:one_step_regret}, we can bound the regret for each episode in the following lemma. 
\begin{lemma}\label{lem:episode_regret}
    Consider the optimistic quasi-Newton update in \eqref{eq:implicit_quasi_newton}. 
    Then for $k=1$ and any $\vu^{1}$ such that $\|\vu^{1}\| \leq D$, we have 
    \begin{equation}\label{eq:k=1_c3}
    \sum_{n=1}^T \langle \vg_n, \vDelta_n - \vu^{1}\rangle \leq \frac{2D^2}{\eta} +   \sum_{n=1}^{T} \eta\|\vg_n - \vh_n\|^2 + {2DT \delta}. 
    \end{equation}
    For $k\geq 2$ and any $\vu^{k}$ such that $\|\vu^{k}\| \leq D$, we have 
    \begin{equation}\label{eq:k>=2_c3}
    \sum_{n=(k-1)T+1}^{kT} \langle \vg_n, \vDelta_n - \vu^{k}\rangle \leq \frac{4D^2}{\eta} + \frac{\eta}{2} \left\| \vg_{(k-1)T} - \vh_{(k-1)T} \right\|^2 +  \sum_{n=(k-1)T}^{kT} \eta\|\vg_n - \vh_n\|^2 + {2DT \delta}. 
    \end{equation}
\end{lemma}

\begin{proof}
To prove \eqref{eq:k=1_c3}, we set $\vu = \vu^{1}$ and sum the inequality in Lemma~\ref{lem:one_step_regret} from $n=1$ to $n = T-1$: 
\begin{align*}
    \sum_{n=2}^T \langle \vg_n, \vDelta_n - \vu^{1}\rangle &\leq \frac{1}{2\eta}\|\vDelta_1-\vu^1\|^2   - \frac{1}{2\eta}\|\vDelta_{T}-\vu^1\|^2 +  \langle \vg_{T} - \vh_{T}, \vDelta_{T} - \vu^1 \rangle -  \langle \vg_{1} - \vh_{1}, \vDelta_{1} - \vu^{1} \rangle \\
    &\phantom{{}={}}+ \sum_{n=1}^{T-1} \left( \eta\|\vg_n - \vh_n\|^2 - \frac{1}{4\eta}\|\vDelta_{n+1} - \vDelta_n\|^2\right) + 2DT \delta. 
\end{align*} 
Moreover, recall that $\vh_1 = \nabla f(\vx_0)$ and $\vDelta_1 = -\frac{D\nabla f(\vx_0)}{\|\nabla f(\vx_0)\|}$, 
which satisfies $\vDelta_1 = \argmin_{\|\vDelta\| \leq D} \langle \vh_1, \vDelta \rangle$. Thus, this implies that $\langle \vh_1, \vDelta_1 - \vu^1\rangle \leq 0$. Furthermore, using the Cauchy-Schwarz and Young's inequalities,  
we can show
\begin{equation*}
    \langle \vg_{T} - \vh_{T}, \vDelta_{T} - \vu^1 \rangle \leq \|\vg_{T} - \vh_{T}\|\|\vDelta_{T} - \vu^1 \|
     \leq \frac{\eta}{2}\|\vg_T-\vh_T\|^2 + \frac{1}{2\eta}\|\vDelta_{T} - \vu^1 \|^2.
\end{equation*}
Combining all the inequalities, we obtain that 
\begin{equation*}
    \sum_{n=1}^T \langle \vg_n, \vDelta_n - \vu^{1}\rangle \leq \frac{1}{2\eta}\|\vDelta_1-\vu^1\|^2  +  \frac{\eta}{2}\|\vg_T-\vh_T\|^2 + \sum_{n=1}^{T-1} \left( \eta\|\vg_n - \vh_n\|^2 - \frac{1}{4\eta}\|\vDelta_{n+1} - \vDelta_n\|^2\right) + {2DT \delta}. 
\end{equation*}
Dropping the negative term $- \frac{1}{4\eta}\|\vDelta_{n+1} - \vDelta_n\|^2$ and using $\frac{\eta}{2}\|\vg_T-\vh_T\|^2 \leq \eta\|\vg_T-\vh_T\|^2 $ lead to: 
\begin{equation*}
    \sum_{n=1}^T \langle \vg_n, \vDelta_n - \vu^{1}\rangle \leq \frac{1}{2\eta}\|\vDelta_1-\vu^1\|^2  +   \sum_{n=1}^{T} \eta\|\vg_n - \vh_n\|^2 + 2DT \delta. 
\end{equation*}
Finally, since we have $\|\vDelta_1\| = D$ and $\|\vu^{1}\| \leq D$, this leads to \eqref{eq:k=1_c3}.

Now we move to \eqref{eq:k>=2_c3}. To simplify the notation, we let $\vDelta_{t}^k :=\vDelta_{(k-1)T+t}$, $\vg_t^k : =  \vg_{(k-1)T+t}$ and $\vh_t^k = \vh_{(k-1)T+t}$. Then by setting $\vu = \vu^{k}$ and summing the inequality in Lemma~\ref{lem:one_step_regret} from $n=(k-1)T$ to $n = kT-1$, we obtain:
\begin{align*}
    \sum_{t=1}^{T} \langle \vg_t^k, \vDelta_t^k - \vu^{k}\rangle &\leq \frac{\|\vDelta^k_1-\vu^k\|^2}{2\eta}   - \frac{\|\vDelta^k_{T}-\vu^k\|^2}{2\eta} +  \langle \vg^k_{T} - \vh^k_{T}, \vDelta^k_{T} - \vu^k \rangle -  \langle \vg^{k-1}_{T} - \vh^{k-1}_{T}, \vDelta_{T-1}^k - \vu^{k} \rangle \\
    &\phantom{{}={}}+ \sum_{n=(k-1)T}^{kT-1} \left( \eta\|\vg_n - \vh_n\|^2 - \frac{1}{4\eta}\|\vDelta_{n+1} - \vDelta_n\|^2\right) + {2DT \delta}.
\end{align*}
Following a similar argument as in the proof of \eqref{eq:k=1_c3}, we have the following inequalities:

\begin{itemize}
    \item $ \langle \vg^k_{T} - \vh^k_{T}, \vDelta^k_{T} - \vu^k \rangle \leq 
    \frac{\eta}{2} \left\| \vg^k_{T} - \vh^k_{T} \right\|^2 + \frac{1}{2\eta} \| \vDelta^k_T - \vu^k \|^2$.
    \item $ \langle \vg^{k-1}_{T} - \vh^{k-1}_{T}, \vDelta^k_{T-1} - \vu^k \rangle \leq 
    \frac{\eta}{2} \left\| \vg^{k-1}_{T} - \vh^{k-1}_{T} \right\|^2 + \frac{1}{2\eta} \| \vDelta^k_{T-1} - \vu^k \|^2$. 
    \item $  \sum_{n=(k-1)T}^{kT-1} \left( \eta\|\vg_n - \vh_n\|^2 - \frac{1}{4\eta}\|\vDelta_{n+1} - \vDelta_n\|^2\right) \leq  \sum_{n=(k-1)T}^{kT-1}  \eta\|\vg_n - \vh_n\|^2$.  
\end{itemize}
 Thus, combining all the inequalities above, we obtain:
\begin{align*}
    \sum_{t=1}^{T} \langle \vg_t^k, \vDelta_t^k - \vu^{k}\rangle &\leq \frac{1}{2\eta}\|\vDelta^k_1-\vu^k\|^2   + \frac{\eta}{2} \left\| \vg^k_{T} - \vh^k_{T} \right\|^2 + \frac{\eta}{2} \left\| \vg^{k-1}_{T} - \vh^{k-1}_{T} \right\|^2 + \frac{1}{2\eta} \| \vDelta^k_{T-1} - \vu^k \|^2\\
    & \phantom{{}={}}+  \sum_{n=(k-1)T}^{kT-1} \eta\|\vg_n - \vh_n\|^2 + {2DT \delta} \\
    & \leq \frac{\|\vDelta^k_1-\vu^k\|^2}{2\eta} + \frac{\eta\left\| \vg^{k-1}_{T} - \vh^{k-1}_{T} \right\|^2}{2}  + \frac{\| \vDelta^k_{T-1} - \vu^k \|^2}{2\eta} +  \sum_{n=(k-1)T}^{kT} \eta\|\vg_n - \vh_n\|^2 \\
    &\phantom{{}={}}+ {2DT \delta}. 
\end{align*}
Finally, since we have $\|\vDelta_1^k\| \leq  D$, $ \|\vu\| \leq D$, and $\|\vDelta_{T-1}^k\| \leq  D$, we get
$\frac{\|\vDelta^k_1-\vu^k\|^2}{2\eta} + \frac{\| \vDelta^k_{T-1} - \vu^k \|^2}{2\eta} \leq \frac{4D^2}{\eta}$. This completes the proof of~\eqref{eq:k>=2_c3}.
\end{proof}

Now we are ready to prove Lemma~\ref{lem:learning_regret}. 
\begin{proof}[Proof of Lemma~\ref{lem:learning_regret}]
Summing the inequality in \eqref{eq:k=1_c3} and the inequality in \eqref{eq:k>=2_c3} for $k = 2,3,\dots,T$ in Lemma~\ref{lem:episode_regret}, 
we have:
\begin{align*}
    \Reg_T(\vu^1,\dots,\vu^K) &= \sum_{k=1}^K \sum_{n=1}^{T} \langle \vg^k_n, \vDelta^k_n - \vu^{k}\rangle \\
     & \leq \frac{4KD^2}{\eta} + \sum_{k=1}^K \frac{\eta}{2} \left\| \vg^{k-1}_{T} - \vh^{k-1}_{T} \right\|^2 +  \sum_{n=1}^{kT} \eta\|\vg_n - \vh_n\|^2 + {2DKT \delta} \\
    & \leq \frac{4KD^2}{\eta} + \frac{3\eta}{2} \sum_{n=1}^{kT} \|\vg_n - \vh_n\|^2 + {2DKT \delta} \\
    & = \frac{4KD^2}{\eta} + \frac{3\eta}{2} \sum_{n=1}^{kT} \left\| \vy_n - \mB_n \vs_n \right\|^2 + {2DKT \delta}.
\end{align*}
This completes the proof. 
\end{proof}

\section{Complete Version and Proof of Theorem \texorpdfstring{\ref{thm:convergence_rate}}{4.1}}

In this section, we first present the complete version of Theorem~\ref{thm:convergence_rate}, where we provide the precise values of the algorithm parameters \( D \), \( \eta \), \( T \), and \( \delta \) (including absolute constants). We also report the complete version of the final upper bound that includes all terms with absolute constants, and even the non-dominant terms are reported for completeness. Then, we present the proof of this theorem.
 
\begin{theorem}\label{thm:convergence_rate_formal}
    Suppose Assumptions \ref{asm:L1} and \ref{asm:L2} hold. If we run  Algorithm~\ref{alg:conversion} with parameters
    $D = \left( \frac{f(\vx_0) - f^*}{52 d^{\frac{2}{5}} L_1^{\frac{2}{5}} L_2^{\frac{3}{5}} M} \right)^{\frac{5}{13}}, \eta = \left( \frac{1} {24dL_1L_2^{\frac{2}{3}}D^{\frac{2}{3}}} \right)^{\frac{3}{5}},  T = \frac{3}{(DL_2\eta)^{\frac{1}{3}}}, \text{and } \delta = \frac{D}{\eta T}$,  we get:
    \begin{align*}
        \frac{1}{K}\sum_{k=1}^K \|\nabla f(\bar{\vw}^k)\| &\leq
        2 (f(\vx_0) - f^*)^{\frac{8}{13}} (52 L_1^{\frac{2}{5}} L_2^{\frac{3}{5}})^{\frac{5}{13}} \frac{d^{\frac{2}{13}}}{M^{\frac{8}{13}}}
        + \frac{L_2(f(\vx_0) - f^*)}{416dL_1M}
        +\frac{7 L_1^{\frac{17}{13}} (f(\vx_0) - f^*)^{\frac{3}{13}}}{L_2^{\frac{7}{13}}} \frac{d^{\frac{4}{13}}}{M^{\frac{16}{13}}}\\
        & \phantom{{}={}} + \frac{L_2^{\frac{7}{13}}}{48} \left( \frac{f(\vx_0) - f^*}{52 d^{\frac{2}{5}} L_1^{\frac{2}{5}} M} \right)^{\frac{10}{13}}. 
\end{align*}

\end{theorem}

\subsection{Proof of Lemma \texorpdfstring{\ref{lem:Dyn_regret_bound}}{4.2}}

To follow the approach outlined in the proof sketch, we first need to upper bound the cumulative loss \(\sum_{n=1}^{KT} \ell_n(\mB_n)\) associated with our projection-free online learning method for updating the Hessian approximation. This is the primary claim of Lemma~\ref{lem:Dyn_regret_bound}. To achieve this, we begin by establishing two intermediate lemmas, i.e., Lemma~\ref{lem:surrogate_regret} and Lemma~\ref{lem:grad_loss}. Before stating these results, recall that 
\begin{equation}\label{eq:def_ln}
    \ell_n (\mB_n) =  \left\| \vy_n -  \mB_n \vs_n \right\|^2 \quad \text{and} \quad \mathcal{Z} = \{\mB\in \sS^d: \|\mB\|_{\op} \leq L_1\}.
\end{equation} 
{In addition, $\{\mW_n\}_{n\geq 0}$ is the auxiliary sequence used in Subroutine~\ref{alg:hessian_approx},  and $\gamma_n$ and $\mS_n$ denote the output of $\mathsf{SEP}(\mW_n)$. Moreover, if $\gamma_n \leq 1$, then $\mB_n \leftarrow \mW_n$ and $\tilde{\mG}_n \leftarrow \nabla \ell_n(\mB_n)$. Otherwise, if $\gamma_n > 1$, then $\mB_n \leftarrow \mW_n/\gamma_n$ and $\tilde{\mG}_n \leftarrow \nabla \ell_n(\mB_n)+\max\{0,-\langle \nabla \ell_n(\mB_n), \mB_n \rangle\} \mS_n$. }

\begin{lemma}\label{lem:surrogate_regret}
Let $\{\mB_n\}$ be generated by Subroutine~\ref{alg:hessian_approx}. Then we have $\|\mB_n\|_{\op} \leq 2L_1$. Moreover, for any $\mB$ such that $\|\mB\|_{\op} \leq L_1$, we have  
    \begin{align}
    \langle \mG_n, {\mB}_n-{\mB} \rangle \leq \langle \tilde{\mG}_n, {\mW}_n-{\mB} \rangle &\leq \frac{1}{2\rho}\|\mW_n-{\mB}\|_F^2-\frac{1}{2\rho}\|\mW_{n+1}-{\mB}\|_F^2+\frac{\rho}{2}\|\tilde{\mG}_n\|_F^2, \label{eq:linearized_loss_matrix} \\
\|\tilde{\mG}_n\|_F &\leq \|\mG_n\|_F + |\langle \mG_n, {\mB}_n\rangle|\|\mS_n\|_F \leq 2\|\mG_n\|_*. \label{eq:surrogate_gradient_bound}
  \end{align}
\end{lemma}

\begin{proof}
We consider two cases depending on the value of $\gamma_n$ returned by $\mathsf{SEP}(\mW_n)$. 
\begin{enumerate}[(a)]
    \item In the first case where $\gamma_n \leq 1$, we have $\mB_n = \mW_n$ and $\tilde{\mG}_n = \mG_n$. Moreover, it holds that $\|\mB_n\|_{\op} = \|\mW_n\|_{\op} \leq 2L_1$ by Definition~\ref{def:extevec}. Thus, $\langle \mG_n, {\mB}_n-{\mB} \rangle = \langle \tilde{\mG}_n, {\mW}_n-{\mB} \rangle$ and $\|\tilde{\mG}_n\|_F = \|\mG_n\|_F$. 
    \item In the second case where $\gamma_n > 1$, we have $\mB_n = \mW_n/\gamma_n$ and $\tilde{\mG}_n = \mG_n +\max\{0,-\langle \mG_n, \mB_n \rangle\} \mS_n$. Moreover, by Definition~\ref{def:extevec} it holds that $\|\mB_n\|_{\op} = \|\mW_n/\gamma_n\|_{\op} \leq 2 L_1$ and $\langle \mS_n,\mW_n -{\mB} \rangle \geq \gamma_n -1$ for any ${\mB}$ such that $\|{\mB}\|_{\op} \leq L_1$. Therefore, 
    \begin{align*}
        \langle \tilde{\mG}_n, \mW_n - \mB \rangle &= \langle \mG_n, \mW_n - \mB \rangle +   \max\{0,-\langle \mG_n, \mB_n \rangle\} \langle \mS_n, \mW_n - \mB\rangle \\
        &= \langle \mG_n, \mB_n - \mB \rangle + (\gamma_n-1)\langle \mG_n, \mB_n \rangle +   \max\{0,-\langle \mG_n, \mB_n \rangle\} \langle \mS_n, \mW_n - \mB\rangle \\
        &\geq \langle \mG_n, \mW_n - \mB \rangle, 
    \end{align*}
    where we used the fact that $\langle \mS_n,\mW_n -{\mB} \rangle \geq \gamma_n -1$ in the last inequality. Furthermore, it follows from the triangle inequality that $\|\tilde{\mG}_n\|_F \leq \|\mG_n\|_F + |\langle \mG_n, {\mB}_n\rangle|\|\mS_n\|_F$, which proves~\eqref{eq:surrogate_gradient_bound}. 
    
\end{enumerate}
Finally, we prove the last inequality in \eqref{eq:linearized_loss_matrix}, which follows from the standard online gradient descent analysis. By using the property of Euclidean projection, we have 
\begin{equation*}
    \langle \mW_{t+1} - \mW_n + \rho \tilde{\mG}_n, \mB - \mW_{t+1} \rangle \geq 0.  
\end{equation*}
Hence, this leads to 
\begin{align*}
    \langle \tilde{\mG}_n, \mW_n - \mB \rangle &\leq \frac{1}{\rho} \langle \mW_{t+1} - \mW_n, \mB- \mW_{t+1}\rangle + \langle \tilde{\mG}_n, \mW_n - \mW_{t+1}\rangle \\ 
    &\leq \frac{1}{2\rho} \|\mW_n - \mB\|_F^2 - \frac{1}{2\rho}\|\mW_{t+1} - \mB\|_F^2 - \frac{1}{2\rho} \|\mW_n - \mW_{t+1}\|_F^2 + \langle \tilde{\mG}_n, \mW_n - \mW_{t+1}\rangle \\
    &\leq \frac{1}{2\rho} \|\mW_n - \mB\|_F^2 - \frac{1}{2\rho}\|\mW_{t+1} - \mB\|_F^2 + \frac{\rho}{2} \|\tilde{\mG}_n\|_F^2. 
\end{align*}
This completes the proof. 
\end{proof}

We first present the following lemma showing a self-bounding property of the loss function $\ell_n$. 
\begin{lemma}\label{lem:grad_loss}
    Recall the definition of $\ell_n:\sS^d \rightarrow \reals$ in \eqref{eq:def_ln}. For any $\mB \in \sS^d$, we have $\| \nabla \ell_n (\mB) \|_* \leq 2D \sqrt{\ell_n(\mB)}$. 
\end{lemma}
\begin{proof}
    By direct calculation, we have 
    $\nabla \ell_n (\mB) = - \left( \vy_n - \mB \vs_n \right) \vs_n^\top - \vs_n \left( \vy_n - \mB \vs_n \right)^\top$.
    Taking the nuclear norm and using the triangle inequality, we have 
    \begin{align*}
        \| \nabla \ell_n (\mB) \|_* & =  \left\|  \left( \vy_n - \mB \vs_n \right) \vs_n^\top + \vs_n \left( \vy_n - \mB \vs_n \right)^\top   \right\|_* \\
        & \leq \left\|  \left( \vy_n - \mB \vs_n \right) \vs_n^\top \right\|_* + \left\| \vs_n \left( \vy_n - \mB \vs_n \right)^\top \right\|_* \\ &\leq  2 \left\| \vy_n - \mB \vs_n \right\| \| \vs_n \|  
        \leq 2\|\vs_n\| \sqrt{\ell_n(\mB)}. 
    \end{align*}
    Finally, recall that $\vs_n = \frac{1}{2}(\vDelta_{n+1}-\vDelta_n)$. Since $\|\vDelta_n\| \leq D$ and $\|\vDelta_{n+1}\| \leq D$, we have $\|\vs_n\| \leq D$. This completes the proof.
\end{proof}

Now we move to the proof of Lemma~\ref{lem:Dyn_regret_bound}. 
\begin{proof}[Proof of Lemma~\ref{lem:Dyn_regret_bound}]
Since $\mH_n \in \mathcal{Z}$, it follows from~\eqref{eq:linearized_loss_matrix} in Lemma~\ref{lem:surrogate_regret} that:
\begin{equation}\label{eq:one_step_regret}
    \begin{aligned}
        \langle \mG_n, {\mB}_n-{\mH_n} \rangle &\leq \frac{1}{2\rho}\|\mW_n-{\mH_n}\|_F^2-\frac{1}{2\rho}\|\mW_{n+1}-{\mH_n}\|_F^2+\frac{\rho}{2}\|\tilde{\mG}_n\|_F^2 \\
        &\leq \frac{1}{2\rho}\|\mW_n-{\mH_n}\|_F^2-\frac{1}{2\rho}\|\mW_{n+1}-{\mH_n}\|_F^2+2\rho \|\mG_n\|^2_*, %
      \end{aligned}
\end{equation}
  where we used \eqref{eq:surrogate_gradient_bound} in the last inequality. 
Recall that  $\mG_n = \nabla \ell_n (\mB_n)$. Since $\ell_n$ is convex, we have $\ell_n(\mB_n) - \ell_n(\mH_n) \leq \langle \mG_n, {\mB}_n-{\mH_n} \rangle$. 
Therefore, applying this inequality and summing \eqref{eq:one_step_regret} from $n=1$ to $n=KT$, we get:
\begin{align*}
    &\sum_{n=1}^{KT} (\ell_n(\mB_n) - \ell_n(\mH_n)) \leq 
    \sum_{n=1}^{KT} 2\rho\|\nabla \ell_n(\mB_n)\|_*^2 +  \sum_{n=1}^{KT} \frac{1}{2 \rho} \left(\|\mW_n - \mH_n\|_F^2 - \|\mW_{n+1} - \mH_n\|_F^2\right).
\end{align*}
Now we choose $\rho = \frac{1}{16D^2} $. Using Lemma  \ref{lem:grad_loss}, the first sum on the right-hand side can be bounded as
\begin{equation*}
    \sum_{n=1}^{KT} 2\rho\|\nabla \ell_n(\mB_n)\|_*^2  \leq \sum_{n=1}^{KT} 8\rho D^2\ell_n(\mB_n)  = \frac{1}{2} \sum_{n=1}^{KT} \ell_n(\mB_n). 
\end{equation*}
Moreover, the second sum on the right-hand side can be bounded by:
\begin{align*}
    & \phantom{{}={}} \sum_{n=1}^{KT} \frac{1}{2 \rho} \left(\|\mW_n - \mH_n\|_F^2 - \|\mW_{n+1} - \mH_n\|_F^2\right) \\
    & = 8D^2 \sum_{n=1}^{KT} \left(\|\mW_n - \mH_n\|_F^2 - \|\mW_{n+1} - \mH_n\|_F^2\right)\\
    & \leq 8D^2  \|\mW_{1} - \mH_1\|_F^2 + 8D^2 \sum_{n=1}^{KT} \left(\|\mW_{n+1} - \mH_{n+1}\|_F^2 - \|\mW_{n+1} - \mH_n\|_F^2\right) 
\end{align*}
Furthermore, note that $\|\mW_{n+1} - \mH_{n+1}\|_F^2 - \|\mW_{n+1} - \mH_n\|_F^2 = ( \|\mW_{n+1} - \mH_{n+1}\|_F - \|\mW_{n+1} - \mH_n\|_F ) ( \|\mW_{n+1} - \mH_{n+1}\|_F + \|\mW_{n+1} - \mH_n\|_F )$. By using the triangle inequality, we have $\|\mW_{n+1} - \mH_{n+1}\|_F - \|\mW_{n+1} - \mH_n\|_F \leq \|\mH_{n+1}-\mH_{n}\|_F$. Also, since $\mH_n,\mH_{n+1} \in \mathcal{Z}$, we have $\|\mH_n\|_{F} \leq \sqrt{d}\|\mH_n\|_{\op} \leq \sqrt{d}L_1$ and $\|\mH_{n+1}\|_{F} \leq \sqrt{d}L_1$. Together with $\mW_{n+1} \in \mathcal{B}_{\sqrt{d}L_1}(0)$, it follows from the triangle inequality that  
\begin{equation*}
    \|\mW_{n+1} - \mH_{n+1}\|_F \leq \|\mW_{n+1} \|_F + \|\mH_{n+1}\|_F \leq 2\sqrt{d}L_1,\quad \|\mW_{n+1} - \mH_n\|_F \leq 2\sqrt{d}L_1. 
\end{equation*}
Hence, we obtain $\|\mW_{n+1} - \mH_{n+1}\|_F^2 - \|\mW_{n+1} - \mH_n\|_F^2 \leq 4L_1\sqrt{d} \| \mH_{n+1}  - \mH_n \|_F$. 
Thus, overall we get:
\begin{equation*}
   \sum_{n=1}^{KT} \ell_n(\mB_n) \leq 16D^2 \|\mW_{1} - \mH_1\|_F^2 +  2 \sum_{n=1}^{KT} \ell_n(\mH_n)   + 64 L_1 D^2 \sqrt{d} \sum_{n=1}^{KT}  \| \mH_{n+1}  - \mH_{n} \|_F.
\end{equation*}
This completes the proof. 
\end{proof}

\subsection{Proof Lemma \texorpdfstring{\ref{lem:H_n_loss}}{4.3}}
Recall that $\vy_n = \vg_{n+1} - \nabla f(\vz_n) = \nabla f(\vw_{n+1}) - \nabla f(\vz_n)$ and $\vs_n = \vw_{n+1} - \vz_n = \frac{1}{2}(\vDelta_{n+1} - \vDelta_n)$. 
By using the definition of $\ell_n$ (see \eqref{eq:def_ln}), we have 
    \begin{align*}
        \ell_n(\mH_n) =   \left\| \vy_n - \nabla^2 f(\vz_{n}) \vs_n \right\|^2 &= \|\nabla f(\vw_{n+1}) - \nabla f(\vz_n) - \nabla^2 f(\vz_{n})( \vw_{n+1} - \vz_n) \|^2\\
        & \leq \left( \frac{L_2}{2} \left\| \vw_{n+1} - \vz_n \right\|^2 \right)^2\\
        & = \frac{L_2^2}{4} \bigl\|\frac{1}{2}(\vDelta_{n+1} - \vDelta_n)\bigr\|^4 \leq \frac{L_2^2}{4}D^4,
    \end{align*}
where the first inequality follows from Assumption~\ref{asm:L2}, and the second inequality holds because $\|\vDelta_n\| \leq D$ and $\|\vDelta_{n+1}\| \leq D$. 
To prove the second inequality, note that by Assumption~\ref{asm:L2} and the relationship between the Frobenius norm and the operator norm:
    \begin{equation*}
        \| \mH_{n+1}  - \mH_n \|_F  = \| \nabla^2f(\vz_{n+1}) - \nabla^2f(\vz_{n})  \|_F \leq  \sqrt{d} \|\nabla^2f(\vz_{n+1}) - \nabla^2f(\vz_{n})\|_{\op} \leq \sqrt{d}L_2 \|\vz_{n+1} - \vz_n\|.
    \end{equation*}
Recall that $\vz_{n} = \vx_n +\frac{1}{2} \vDelta_n$, $\vz_{n+1} = \vx_{n+1} + \frac{1}{2} \vDelta_{n+1}$ and $\vx_{n+1} = \vx_n + \vDelta_{n+1}$. Thus, $\vz_{n+1} - \vz_n = \vx_{n+1} + \frac{1}{2} \vDelta_{n+1} - \vx_n - \frac{1}{2} \vDelta_n = \frac{3}{2}\vDelta_{n+1}- \frac{1}{2} \vDelta_n$. Since  $\| \vDelta_n \| \leq D$ and $\| \vDelta_{n+1} \| \leq D$, we further have $\|\vz_{n+1} - \vz_n\| \leq 2D$. This completes the proof.

\subsection{The Choices of Hyperparameters}

The logic behind selecting the hyperparameters is to balance the dominant terms in the upper bound, thereby optimizing the parameters to achieve the lowest possible convergence bound. 
Recall the following inequality from the last equation in the Proof Sketch: 
\begin{equation*}
        \frac{1}{K}\sum_{k=1}^K \|\nabla f(\bar{\vw}^k)\| \leq \frac{f(\vx_0) - f^*}{DM}
        +\frac{6D}{T\eta} + \frac{3\eta}{2} \left[  \frac{L_2^2 D^3}{2} + \frac{64 dL_1^2D}{M} + 128 d L_1 L_2 D^2  \right]
        + \frac{L_2 D^2}{48} + \frac{L_2D^2T^2 }{2}.
\end{equation*}

To proceed, we first balance the two terms involving $T$ to find an optimal value. This leads to our choice of $T = \frac{3}{(DL_2\eta)^{\frac{1}{3}}}$, and substituting this back into the bound yields:

\begin{equation*}
        \frac{1}{K}\sum_{k=1}^K \|\nabla f(\bar{\vw}^k)\| \leq \frac{f(\vx_0) - f^*}{DM}
        +\frac{13}{2}\frac{D^{\frac{4}{3}} L_2^{\frac{1}{3}}}{\eta^{\frac{2}{3}}} + \frac{3\eta}{2} \left[  \frac{L_2^2 D^3}{2} + \frac{64 dL_1^2D}{M} + 128 d L_1 L_2 D^2  \right]
        + \frac{L_2 D^2}{48}.
\end{equation*}
{Next, we balance the two terms $\frac{13}{2}\frac{D^{\frac{4}{3}} L_2^{\frac{1}{3}}}{\eta^{\frac{2}{3}}}$ and $\frac{3\eta}{2} 128 d L_1 L_2 D^2 $ to obtain $\eta$, as it will be clear later that the second term is the leading term inside the brackets. This yields $\eta = \left( \frac{1} {24dL_1L_2^{\frac{2}{3}}D^{\frac{2}{3}}} \right)^{\frac{3}{5}}$, resulting in the following bound:}
\begin{equation*}
        \frac{1}{K}\sum_{k=1}^K \|\nabla f(\bar{\vw}^k)\| \leq \frac{f(\vx_0) - f^*}{DM}
        + 52(d L_1)^{\frac{2}{5}} D^{\frac{8}{5}} L_2^{\frac{3}{5}}
        + \frac{1}{8} \left( \frac{1}{dL_1} \right)^{\frac{3}{5}} L_2^{\frac{8}{5}} D^{\frac{13}{5}}
        + 15 \frac{d^{\frac{2}{5}} L_1^{\frac{7}{5}} D^{\frac{3}{5}}}{L_2^{\frac{2}{5}} M}
        + \frac{L_2}{48} D^2. 
\end{equation*}
Finally, we balance $D$ between the first and second terms. This results in $D = \left( \frac{f(\vx_0) - f^*}{52 d^{\frac{2}{5}} L_1^{\frac{2}{5}} L_2^{\frac{3}{5}} M} \right)^{\frac{5}{13}}$, leading to:
\begin{align*}
        \frac{1}{K}\sum_{k=1}^K \|\nabla f(\bar{\vw}^k)\| &\leq
        2 (f(\vx_0) - f^*)^{\frac{8}{13}} (52 L_1^{\frac{2}{5}} L_2^{\frac{3}{5}})^{\frac{5}{13}} \frac{d^{\frac{2}{13}}}{M^{\frac{8}{13}}}
        + \frac{L_2(f(\vx_0) - f^*)}{416dL_1M}
        +\frac{7 L_1^{\frac{17}{13}} (f(\vx_0) - f^*)^{\frac{3}{13}}}{L_2^{\frac{7}{13}}} \frac{d^{\frac{4}{13}}}{M^{\frac{16}{13}}}\\
        & + \frac{L_2^{\frac{7}{13}}}{48} \left( \frac{f(\vx_0) - f^*}{52 d^{\frac{2}{5}} L_1^{\frac{2}{5}} M} \right)^{\frac{10}{13}}.
\end{align*}
To summarize,  the hyperparameters in our algorithm are chosen as follows (ignoring absolute constants): 

\begin{equation}\label{eq:parameters_cr_exp}
    D = \left( \frac{f(\vx_0) - f^*}{d^{\frac{2}{5}} L_1^{\frac{2}{5}} L_2^{\frac{3}{5}} M} \right)^{\frac{5}{13}},
    \quad
    \quad
    \eta = \frac{M^{\frac{2}{13}}}{(f(\vx_0)-f^*)^{\frac{2}{13}} d^{\frac{7}{13}} L_1^{\frac{7}{13}} L_2^{\frac{4}{13}}},
    \quad
    \quad
    T = \frac{(d L_1)^{\frac{3}{13}} M^{\frac{1}{13}}}{(f(\vx_0-f^*))^{\frac{1}{13}} L_2^{\frac{2}{13}}}. 
\end{equation}

\section{Implementation of \texorpdfstring{$\mathsf{TRSolver}$}{TRSolver}}
\label{appen:trsolver}

In this section, we describe the implementation details of the $\mathsf{TRSolver}$ oracle defined in Definition~\ref{def:TRSolver}, which approximately solves the trust-region subproblem in~\eqref{eq:trust_region_generic}.  
As mentioned in Section~\ref{subsec:computational}, our first step is to approximately compute
the minimum eigenvalue of $\mA$ to determine whether the problem is convex. 
To this end, we define the following oracle:  
\begin{definition}\label{def:minevec}
    The $\mathsf{MinEvec}(\mA; \delta)$ oracle  takes $\mA \in \mathbb{S}^d$ and $\delta>0$ %
    as inputs. 
    It outputs $\hat{\lambda}_{\min} \in \reals$ and $\hat{\vv}_{\min} \in \reals^d$ such that
    one of the following outcomes holds: 
    \begin{enumerate}[(a)]
        \item $\hat{\lambda}_{\min} \geq 0$, which implies that $\lambda_{\min}(\mA) \geq \hat{\lambda}_{\min} \geq 0$; 
        \item $\hat{\lambda}_{\min}<0$, which implies that $\hat{\lambda}_{\min} + \delta \geq \lambda_{\min}(\mA) \geq \hat{\lambda}_{\min}$ %
        and $\|\mA \hat{\vv}_{\min} - \hat{\lambda}_{\min} \hat{\vv}_{\min}\| \leq \delta$. 
    \end{enumerate}
\end{definition}
To summarize, the $\mathsf{MinEvec}$ oracle has two possible outcomes. 
In the first case, we have $\hat{\lambda}_{\min} \geq 0$, which certifies that the input matrix $\mA$ is positive semidefinite (PSD). In the second case, we have $\hat{\lambda}_{\min} < 0$, and further we guarantee that the actual minimum eigenvalue $\lambda_{\min}(\mA)$ lies in the interval $[\hat{\lambda}_{\min}, \hat{\lambda}_{\min}+\delta]$ and $\hat{\vv}$ is an approximate eigenvector such that $\|\mA \hat{\vv} - \hat{\lambda}_{\min} \hat{\vv}\| \leq \delta$. 
We implement the $\mathsf{MinEvec}$ oracle based on the Lanczos method with a random start~\cite{kuczynski1992estimating} and the details are given in Appendix~\ref{appen:lanczos}. For now, we present the following proposition that summarizes the computational cost of the $\mathsf{MinEvec}$ oracle.  

\begin{proposition}\label{prop:minevec}
    Given an input matrix $\mA \in \sS^d$, suppose $B>0$ is an upper bound on $\lambda_{\max}(\mA) - \lambda_{\min}(\mA)$. Then the $\mathsf{MinEvec}(\mA;\delta)$ oracle can be implemented, with success probability at least $1-q$, using at most $\lceil\frac{1}{4}\sqrt{\frac{2B}{\delta}}\log(\frac{44d B}{q^2 \delta})+\frac{1}{2} \rceil$ matrix-vector products.
\end{proposition}

Another building block of our implementation is a fast algorithm for solving convex-constrained optimization problems. Specifically, consider the minimization problem 
\begin{equation}\label{eq:constrained}
    \min_{\vx \in Q} g(\vx),
\end{equation}
where $g$ is a convex function with $L_g$-Lipschitz gradients and $Q$ is a closed convex set. 
Then we have the following convergence results based on \cite{lee2021geometric,kim2023time} and we defer the details to Appendix~\ref{appen:gradient_norm}. 
\begin{proposition}\label{prop:fista_sfg}
Suppose $g: \reals^d \rightarrow \reals$ is convex with $L_g$-Lipschitz gradients and $Q$ is closed and convex. Let $\vx^*$ denote the optimal solution of \eqref{eq:constrained}. 
There exists an algorithm $\mathsf{FISTA\mathrm{+}SFG}$ that, initialized at $\vx_0$, finds $\hat{\vx} \in Q$ satisfying 
$\min_{\vu \in \mathcal{N}_Q(\hat{\vx})}\|\nabla g(\hat{\vx}) + \vu\| \leq \delta$ after at most 
$2\sqrt{\frac{10 L_g \|\vx_0-\vx^*\|}{\delta}}$ gradient queries and projections onto $Q$. 
\end{proposition}

\begin{subroutine}[!t]\small
    \caption{$\mathsf{TRSolver}(\mA,\vb,D;\delta)$}\label{alg:TRSolver}
    \begin{algorithmic}[1]
        \State \textbf{Input:} $\mA \in \mathbb{S}^d$, $\vb \in \reals^d$, $D>0$, $\delta>0$ %
        \State Set $(\hat{\lambda}_{\min}, \hat{\vv}_{\min}) = \mathsf{MinEvec}(\mA;\frac{\delta}{2D})$
        \If{$\hat{\lambda}_{\min} \geq 0$}
        \State Run $\mathsf{FISTA\mathrm{+}SFG}$ on Problem~\eqref{eq:trust_region_generic}
        to find $\tilde{\vDelta} \in B_D(0)$ with $\min_{\vv \in \mathcal{N}(\tilde{\vDelta})}\|\mA \tilde{\vDelta} + \vb + \vv\| \leq  \delta$
        \State Return $\hat{\vDelta} \leftarrow \tilde{\vDelta}$
        \Else
        \State Run $\mathsf{FISTA\mathrm{+}SFG}$ on Problem~\eqref{eq:regularized}
        to find $\tilde{\vDelta}\in B_D(0)$ with $\min_{\vv \in \mathcal{N}(\tilde{\vDelta})}\|(\mA- \hat{\lambda}_{\min}\mI)\tilde{\vDelta} + \vb + \vv\| \leq  \frac{\delta}{2}$
        \If{$\|\tilde{\vDelta}\| = D$}
        \State Return $\hat{\vDelta} \leftarrow \tilde{\vDelta}$
        \Else 
        \State Compute $\alpha = \sqrt{(\tilde{\vDelta}^\top \hat{\vv}_{\min})^2 + (D^2 - \|\tilde{\vDelta}\|^2)} - \tilde{\vDelta}^\top \hat{\vv}_{\min}$
        \State Return $\hat{\vDelta} \leftarrow \tilde{\vDelta} + \alpha \hat{\vv}_{\min}$ \quad \Comment{The choice of $\alpha$ ensures that $\|\hat{\vDelta}\| = D$}
        \EndIf
        \EndIf
    \end{algorithmic}
  \end{subroutine}

Now we are ready to describe our procedure for solving \eqref{eq:trust_region_generic}, which is presented in Subroutine~\ref{alg:TRSolver}. Specifically, we first call the $\mathsf{MinEvec}(\mA; \frac{\delta}{2D})$ oracle to obtain the approximate eigenvalue $\hat{\lambda}_{\min}$ and the approximate eigenvector $\hat{\vv}_{\min}$. Depending on the sign of $\hat{\lambda}_{\min}$, we consider the following two cases:   
\begin{itemize}
    \item If $\hat{\lambda}_{\min} \geq 0$, we are in the first case of Definition~\ref{def:minevec} and this implies that $\lambda_{\min}(\mA) \geq 0$, which certifies that Problem~\eqref{eq:trust_region_generic} is convex. Hence, we run $\mathsf{FISTA\mathrm{+}SFG}$ on Problem~\eqref{eq:trust_region_generic} with $g(\vDelta)= \frac{1}{2}\vDelta^\top \mA \vDelta + \vb^\top \vDelta$, $Q = B_D(0)$ and $\vDelta_0 = 0$. 
    Note that the gradient of $g(\vDelta)$ is $\lambda_{\max}(\mA)$-Lipschitz and $\sup_{\vDelta \in Q}\|\vDelta_0 - \vDelta\| \leq D$. Thus, by Proposition~\ref{prop:fista_sfg}, after at most $2\sqrt{\frac{10 \lambda_{\max}(\mA) D}{\delta}}$ iterations, we can find $\tilde{\vDelta} \in Q$ such that:
    \begin{equation}\label{eq:subgradient_norm}
        \min_{\vv \in \mathcal{N}_Q(\tilde{\vDelta})}\|\mA \tilde{\vDelta} + \vb + \vv\| \leq \delta,
    \end{equation}
    which shows that $\tilde{\vDelta}$ satisfies the requirement in Definition~\ref{def:TRSolver}. 

    \item  If $\hat{\lambda}_{\min} < 0$, we are in the second case of Definition~\ref{def:minevec}, which implies that $\hat{\lambda}_{\min}  \leq \lambda_{\min}(\mA) \leq \hat{\lambda}_{\min} + \frac{\delta}{2D}$ and $\|\mA \hat{\vv}_{\min} - \hat{\lambda}_{\min} \hat{\vv}\| \leq \frac{\delta}{2D}$. Consider the regularized problem:
    \begin{equation}\label{eq:regularized}
        \min_{\|\vDelta\| \leq D} \left\{ \frac{1}{2}\vDelta^\top \left(\mA - \hat{\lambda}_{\min} \mI\right) \vDelta + \vb^\top \vDelta\right\}.
    \end{equation}
    Since $\hat{\lambda}_{\min} \leq \lambda_{\min}(\mA)$, the matrix $\mA - \hat{\lambda}_{\min} \mI$ is PSD and the problem in \eqref{eq:regularized} is convex. Hence, we can similarly run $\mathsf{FISTA\mathrm{+}SFG}$ on Problem~\eqref{eq:regularized} with $Q = \mathcal{B}_D(0)$ and $\vDelta_0 = 0$. Again by Proposition~\ref{prop:fista_sfg}, after at most:
    \begin{align}
    2\sqrt{\frac{20 (\lambda_{\max}(\mA) - \hat{\lambda}_{\min}) D}{\delta}} &\leq 2\sqrt{\frac{20 (\lambda_{\max}(\mA) - {\lambda}_{\min}(\mA) + \frac{\delta}{2D}) D}{\delta}}\nonumber \\
    &= 2\sqrt{\frac{20 (\lambda_{\max}(\mA) - {\lambda}_{\min}(\mA)) D}{\delta}+10}
    \end{align}
    iterations, 
    we can find $\tilde{\vDelta} \in Q$ such that: \begin{equation}\label{eq:subgradient_norm_regularized}
        \min_{\vv \in \mathcal{N}_Q(\tilde{\vDelta})}\|(\mA - \hat{\lambda}_{\min}\mI)\tilde{\vDelta} + \vb + \vv\| \leq \frac{\delta}{2}.
    \end{equation}
    We further consider two subcases {similar to the approach in \cite{honguyen2017second,wang2017linear}.} 
    \begin{itemize}
        \item If $\|\tilde{\vDelta}\| = D$, i.e., $\tilde{\vDelta}$ is on the boundary of $Q$, 
        we can show that $\tilde{\vDelta}$ satisfies the requirement in Definition~\ref{def:TRSolver}.
        Indeed, in this case, $\mathcal{N}_Q(\tilde{\vDelta}) = \{c \tilde{\vDelta}: c\geq 0\}$. Since $\hat{\lambda}_{\min}<0$, we have $ - \hat{\lambda}_{\min}\tilde{\vDelta} + \vv \in \mathcal{N}_Q(\tilde{\vDelta})$ for any $\vv \in \mathcal{N}_Q(\tilde{\vDelta})$. Thus, it follows from \eqref{eq:subgradient_norm_regularized}  that $\min_{\vv \in \mathcal{N}_Q(\tilde{\vDelta})}\|\mA \tilde{\vDelta} + \vb + \vv\| \leq \min_{\vv \in \mathcal{N}_Q(\tilde{\vDelta})}\|(\mA - \hat{\lambda}_{\min}\mI)\tilde{\vDelta} + \vb + \vv\| \leq \frac{\delta}{2}$. 
        \item If $\|\tilde{\vDelta}\| < D$, then we set $\hat{\vDelta} \leftarrow \tilde{\vDelta} + \alpha \hat{\vv}_{\min}$, where $\alpha$ is chosen such that $\|\hat{\vDelta}\| = D$ (note that we can ensure $|\alpha| \leq D$). We claim that $\hat{\vDelta}$ satisfies the requirement in Definition~\ref{def:TRSolver}. Indeed, since $\tilde{\vDelta}$ is in the interior of $D$, we have $\mathcal{N}_Q(\tilde{\vDelta}) = \{0\}$ and \eqref{eq:subgradient_norm_regularized} becomes $\|(\mA - \hat{\lambda}_{\min}\mI)\tilde{\vDelta} + \vb \| \leq \frac{\delta}{2}$. Moreover, we can compute: 
        \begin{align*}
            \min_{\vv \in \mathcal{N}_Q(\hat{\vDelta})}\|\mA \hat{\vDelta} + \vb + \vv\| &\leq  \|(\mA - \hat{\lambda}_{\min}\mI) \hat{\vDelta} + \vb\| \qquad (\text{since } - \hat{\lambda}_{\min}\hat{\vDelta}\in \mathcal{N}(\hat{\vDelta})) \\
            & \leq \|(\mA - \hat{\lambda}_{\min}\mI) \tilde{\vDelta} + \vb\| + \alpha \|(\mA - \hat{\lambda}_{\min}\mI)\hat{\vv}_{\min}\| \\
            &\leq \frac{\delta}{2} + D \cdot \frac{\delta}{2D} = \delta. 
        \end{align*}
    \end{itemize}
\end{itemize}
Considering all cases, we can conclude that $\hat{\vDelta}$ returned by Subroutine~\ref{alg:TRSolver} satisfies the condition in Definition~\ref{def:TRSolver}. 
Moreover, together with Proposition~\ref{prop:minevec}, we have the following guarantee on the computational cost of Subroutine~\ref{alg:TRSolver}.

\begin{corollary}\label{coro:TRSolver}
Given an input matrix $\mA \in \sS^d$, suppose $B>0$ is an upper bound on $\lambda_{\max}(\mA) - \lambda_{\min}(\mA)$. Then we can implement the $\mathsf{TRSolver}$ oracle with success probability at least $1-q$, with
the total number of matrix-vector products bounded by: 
\begin{equation*}
    \biggl\lceil\frac{1}{2}\sqrt{\frac{BD}{\delta}}\log(\frac{88d BD}{q^2 \delta})+\frac{1}{2} \biggr\rceil + \max\biggl\{2\sqrt{\frac{10 \lambda_{\max}(\mA) D}{\delta}},  2\sqrt{\frac{20 B D}{\delta}+10}\biggl\}.
\end{equation*}
\end{corollary}
\begin{proof}
    The first term corresponds to the computational cost of calling $\mathsf{MinEvec}(\mA;\frac{\delta}{2D})$, as established in Proposition~\ref{prop:minevec}. The second term corresponds to the cost of running the $\mathsf{FISTA\mathrm{+}SFG}$ algorithm. As discussed above, in the first case, this cost is bounded by $2\sqrt{\frac{10 \lambda_{\max}(\mA) D}{\delta}}$, while in the second case, it is bounded by $2\sqrt{\frac{20 B D}{\delta}+10}$.  
\end{proof}

\subsection{Implementation of \texorpdfstring{$\mathsf{MinEvec}$}{MinEvec}
}\label{appen:lanczos}

Our goal in this section is to describe the implementation of the $\mathsf{MinEvec}$ oracle in Definition~\ref{def:minevec} and to prove Proposition~\ref{prop:minevec}. 
As mentioned in the previous section, our approach is based on the Lanczos method with a random start~\cite{kuczynski1992estimating}, which approximates eigenvectors and eigenvalues within a Krylov subspace. Hence, we first recall the following lemma that characterizes the convergence property of the Lanczos process.  

\begin{lemma}[{\cite[Theorem 4.2]{kuczynski1992estimating}}]\label{prop:lanczos_estimate}
    Consider a symmetric matrix $\mA \in \sS^d$ and let $\lambda_{\max}(\mA)$ and $\lambda_{\min}(\mA)$ be its maximum and minimum eigenvalues, respectively. Let $\vu \in \reals^d$ be a random vector drawn uniformly from the unit sphere. 
    Define the Krylov subspace $\mathcal{K}_t(\mA,\vu)$ as: 
    \begin{equation}\label{eq:krylov}
        \mathcal{K}_t(\mA,\vu) = \mathrm{span}\{ \vu, \mA\vu, \dots, \mA^{t-1} \vu\}. 
    \end{equation}    
    Then we have: 
    \begin{align*}
        \Pr\Bigl(\min_{\vv \in \mathcal{K}_t(\mA,\vu)} \frac{\vv^\top \mA \vv}{\vv^\top \vv} \geq \lambda_{\min}(\mA) + \rho (\lambda_{\max}(\mA) - \lambda_{\min}(\mA))\Bigr) &\leq 1.648 \sqrt{d} e^{-\sqrt{\rho}(2t-1)}, \\
        \Pr\Bigl(\max_{\vv \in \mathcal{K}_t(\mA,\vu)} \frac{\vv^\top \mA \vv}{\vv^\top \vv} \leq \lambda_{\max}(\mA) - \rho (\lambda_{\max}(\mA) - \lambda_{\min}(\mA))\Bigr) &\leq 1.648 \sqrt{d} e^{-\sqrt{\rho}(2t-1)}. 
    \end{align*}
\end{lemma}

Lemma~\ref{prop:lanczos_estimate} describes how the Rayleigh quotient, $\frac{\vv^\top \mA \vv}{\vv^\top \vv}$, converges to the extreme eigenvalues of~$\mA$ as the order of the Krylov subspace increases.  
As will be evident later, we also require an additional result that characterizes the convergence behavior of $\frac{\|\mA\vv\|}{\|\vv\|}$. 
Since we could not find such a result in the existing literature, we follow the analysis in \cite{kuczynski1992estimating} to derive the following lemma.

\begin{lemma}\label{lem:lanczos_new}
    Consider a symmetric matrix $\mA \in \sS^d$ and let $\lambda_{\max}(\mA)$ and $\lambda_{\min}(\mA)$ be its maximum and minimum eigenvalues, respectively. Let $\vu$ be a random vector drawn uniformly from the unit sphere and recall the definition of the Krylov subspace from \eqref{eq:krylov}. Then we have:    
    \begin{equation*}
        \Pr\left(\min_{\vv \in \mathcal{K}_t(\mA,\vu)} \frac{\|\mA \vv\|}{\|\vv\|} \geq \lambda_{\min}(\mA) + \rho (\lambda_{\max}(\mA) - \lambda_{\min}(\mA)) \right) \leq 2.34 \sqrt{\frac{d}{\rho}} e^{-\sqrt{\rho}(2t-1)}.
    \end{equation*}
\end{lemma}
\begin{proof}
We aim to upper bound $\min_{\vv \in \mathcal{K}_t(\mA,\vu)} \frac{\|\mA \vv\|^2}{\|\vv\|^2} = \min_{\vv \in \mathcal{K}_t(\mA,\vu)} \frac{\vv^\top \mA^2 \vv}{\vv^\top \vv}$. 
Since $\vv \in \mathcal{K}_t(\mA,\vu)$, we let $\vv = P(\mA) \vu$, where $P$ can be any non-zero polynomial of degree less than or equal to $t-1$. Also, let $\lambda_1,\lambda_2,\dots,\lambda_d$ denote the eigenvalues of $\mA$ in decreasing order. Then: 
\begin{equation*}
    \frac{\vv^\top \mA^2 \vv}{\vv^\top \vv} = \frac{\sum_{i=1}^d u_i^2 \lambda_i^2 P^2(\lambda_i)}{\sum_{i=1}^d u_i^2 P^2(\lambda_i)}, 
\end{equation*}
and we further have:  
\begin{equation*}
    \frac{1}{(\lambda_1 - \lambda_d)^2}\left(\frac{\vv^\top \mA^2 \vv}{\vv^\top \vv} - \lambda_d^2\right) = \frac{1}{(\lambda_1 - \lambda_d)^2} \frac{\sum_{i=1}^d u_i^2 (\lambda_i^2 - \lambda_d^2) P^2(\lambda_i)}{\sum_{i=1}^d u_i^2 P^2(\lambda_i)}. 
\end{equation*}
Note that for any $i=1,2,\dots,d$, we can write: 
$\lambda^2_i - \lambda_d^2 = (\lambda_i - \lambda_d)^2 + 2\lambda_d (\lambda_i -\lambda_d)$. 
Thus, the above further becomes: 
\begin{equation}\label{eq:two_ratios}
    \frac{\sum_{i=1}^d u_i^2 \left(\frac{\lambda_i - \lambda_d}{\lambda_1 - \lambda_d}\right)^2 P^2(\lambda_i)}{\sum_{i=1}^d u_i^2 P^2(\lambda_i)} +  \frac{2\lambda_d}{\lambda_1 - \lambda_d} \frac{\sum_{i=1}^d u_i^2 \frac{\lambda_i - \lambda_d}{\lambda_1 - \lambda_d} P^2(\lambda_i)}{\sum_{i=1}^d u_i^2 P^2(\lambda_i)}. 
\end{equation}
Now define the auxiliary variables $x_i = \frac{\lambda_1 - \lambda_i}{\lambda_1 - \lambda_d} \in (0,1]$ for $i\in \{1,2,\dots,d\}$ and $Q(x) = \frac{P(\lambda_1 - (\lambda_1 - \lambda_d)x)}{P(\lambda_d)}$. It is easy to see that $x_d = 1$ and $Q$ is a polynomial of degree less or equal to $t-1$ satisfying $Q(1) = 1$. Moreover, under this notation, we have $\frac{\lambda_i - \lambda_d}{\lambda_1 - \lambda_d} = 1-x_i$, $x_d = 1$, and $P(\lambda_i) = Q(x_i) P(\lambda_d)$. 
Thus, for any $\rho < 1$, we can upper bound the first term in \eqref{eq:two_ratios} as: 
\begin{align*}
    \frac{\sum_{i=1}^{d} u_i^2 (1-x_i)^2 Q^2(x_i)}{\sum_{i=1}^d u_i^2 Q^2(x_i)}  &= \rho^2 + \frac{\sum_{i=1}^d u_i^2 ((1-x_i)^2-\rho^2) Q^2(x_i)}{\sum_{i=1}^d u_i^2 Q^2(x_i)} \\
     &=  \rho^2 + \frac{\sum_{i=1}^{d-1} u_i^2 ((1-x_i)^2-\rho^2) Q^2(x_i) - \rho^2 u_d^2}{\sum_{i=1}^d u_i^2 Q^2(x_i)} \\
    & \leq \rho^2 + \frac{\sum_{i=1}^{d-1} 2u_i^2 (1-\rho-x_i) Q^2(x_i) - \rho^2 u_d^2}{\sum_{i=1}^d u_i^2 Q^2(x_i)} \\
    & \leq \rho^2 + \frac{2\sum_{i=1}^{d-1} u_i^2 \max_{0\leq x \leq 1-\rho }\{(1-\rho-x)Q^2(x)\} - \rho^2 u_d^2}{ u_d^2}, 
\end{align*}
where we used $x_d = 1$ and $Q(1) = 1$ in the second equality, the first inequality is due to the fact that $(1-x_i)^2 - \rho^2 = (1+\rho-x_i)(1-\rho-x_i)\leq 2(1-\rho-x_i)$, and the last inequality is because $x_i \in [0,1]$ and $(1-\rho-x_i) Q^2(x_i) \leq 0$ when $x_i > 1-\rho$. Similarly, we can upper bound the second term in 
\eqref{eq:two_ratios} as:
\begin{equation*}
    \frac{2\lambda_d}{\lambda_1 - \lambda_d} \frac{\sum_{i=1}^{d} u_i^2 (1-x_i) Q^2(x_i)}{\sum_{i=1}^d u_i^2 Q^2(x_i)}\leq  \frac{2\lambda_d}{\lambda_1 - \lambda_d} \rho + \frac{2\lambda_d}{\lambda_1-\lambda_d} \frac{\sum_{i=1}^{d-1} u_i^2\max_{0\leq x \leq 1-\rho }\{(1-\rho-x)Q^2(x)\}   -\rho u_d^2}{u_d^2}. 
\end{equation*}
We recall the following two helper lemmas from \cite[Theorem 4.2]{kuczynski1992estimating}. 
\begin{lemma}\label{lem:poly_approx}
    There exists a polynomial $T$ of degree $t-1$ and $T(1) = 1$ such that: 
    \begin{equation*}
        \max_{0 \leq x \leq 1-\rho} \{(1-\rho-x)T^2(x)\} \leq 4\rho \Bigl( \frac{1-\sqrt{\rho}}{1+\sqrt{\rho}}\Bigr)^{2t-1} \Bigl( 1- \Bigl( \frac{1-\sqrt{\rho}}{1+\sqrt{\rho}}\Bigr)^{2t-1} \Bigr)^{-2} =  \frac{4\rho \gamma}{(1-\gamma)^2},
    \end{equation*}
    where we define $\gamma :=   \left( \frac{1-\sqrt{\rho}}{1+\sqrt{\rho}}\right)^{2t-1}$ to simplify the notation.
\end{lemma}

\begin{lemma}\label{lem:prob}
    Let $\vu \in \reals^d$ be a random vector drawn from the unit sphere and let $\vu_i$ denote its $i$-th coordinate ($1 \leq i \leq d$). Then $\Pr(\sum_{i=1}^{d-1} u_i^2 > c u_d^2) \leq 0.824\sqrt{\frac{d}{1+c}}$. 
\end{lemma}
Recall that $P$ can be any non-zero polynomial of degree less than or equal to $t-1$. Specifically, we choose $P(\lambda) = T(\frac{\lambda_1 - \lambda}{\lambda_1-\lambda_d})$ with the polynomial $T$ given in Lemma~\ref{lem:poly_approx}, which corresponds to $Q(x) = T(x)$. Thus, we obtain: 
\begin{equation}\label{eq:before_probability}
    \min_{\vv\in \mathcal{K}_t(\mA,\vu)}\frac{1}{(\lambda_1 - \lambda_d)^2}\left(\frac{\vv^\top \mA^2 \vv}{\vv^\top \vv} - \lambda_d^2\right) \leq \rho^2 + \frac{2\lambda_d}{\lambda_1 - \lambda_d} \rho + \frac{2\lambda_1}{\lambda_1 - \lambda_d}\frac{\frac{8\rho\gamma}{(1-\gamma)^2}\sum_{i=1}^{d-1}u_i^2 - \rho^2 u_d^2}{u_d^2}.
\end{equation}
Moreover, by Lemma~\ref{lem:prob}, it holds that: 
\begin{equation*}
    \Pr\left(\frac{8\rho\gamma}{(1-\gamma)^2}\sum_{i=1}^{d-1}u_i^2 > \rho^2 u_d^2\right) \leq 0.824 \sqrt{\frac{d}{1+\frac{\rho(1-\gamma)^2}{8\gamma}}} \leq 0.824 \sqrt{\frac{8d\gamma}{\rho}} \leq 2.34\sqrt{\frac{d}{\rho}} e^{-\sqrt{\rho}(2t-1)},
\end{equation*}
where we used the fact that $1+\frac{\rho(1-\gamma)^2}{8\gamma} = 1+\frac{\rho}{8\gamma} - \frac{\rho}{4} + \frac{\rho\gamma}{8} \geq \frac{3}{4} + \frac{\rho}{8\gamma} \geq \frac{\rho}{8\gamma}$ in the second inequality and the fact that $\sqrt{\gamma} \leq e^{-\sqrt{\rho}(2t-1)}$ in the last inequality. Finally, we note that when $\frac{8\rho\gamma}{(1-\gamma)^2}\sum_{i=1}^{d-1}u_i^2 \leq \rho^2 u_d^2$ holds, we obtain from \eqref{eq:before_probability} that: 
\begin{align*}
    \min_{\vv\in \mathcal{K}_t(\mA,\vu)}\frac{1}{(\lambda_1 - \lambda_d)^2}\left(\frac{\vv^\top \mA^2 \vv}{\vv^\top \vv} - \lambda_d^2\right) &\leq \rho^2 + \frac{2\lambda_d}{\lambda_1 - \lambda_d} \rho \\
    \Leftrightarrow \quad \min_{\vv\in \mathcal{K}_t(\mA,\vu)}\frac{\vv^\top \mA^2 \vv}{\vv^\top \vv}   &\leq \rho^2(\lambda_1 - \lambda_d)^2 + 2\lambda_d\rho(\lambda_1 - \lambda_d) + \lambda_d^2 \\
    \Leftrightarrow \quad \min_{\vv \in \mathcal{K}_t(\mA,\vu)} \frac{\|\mA \vv\|}{\|\vv\|} &\leq \lambda_d + \rho (\lambda_1 - \lambda_d). 
\end{align*}
This completes the proof. 
\end{proof}

\begin{subroutine}[!t]\small
  \caption{$\mathsf{MinEvec}(\mA;\delta,q)$}\label{alg:lanczos}
  \begin{algorithmic}[1]
      \State \textbf{Input:} $\mA \in \mathbb{S}^d$, $\delta>0$, $q\in (0,1)$, an upper bound $B$ on $\lambda_{\max}(\mA) - \lambda_{\min}(\mA)$
      \State \textbf{Initialize:} sample $\vv_1\in \reals^d$ uniformly from the unit sphere, $\beta_1 \leftarrow 0$, $\vv_0\leftarrow 0$
      \State Set $N_1 \leftarrow \lceil\frac{1}{4}\sqrt{\frac{2B}{\delta}}\log(\frac{11d}{q^2})+\frac{1}{2} \rceil$
      \For{$k=1,\dots,N_1$ \tikzmark{top}} \label{line:lanczos_start_minevec}
      \State Set $\vw_k \leftarrow \mA \vv_k-\beta_k \vv_{k-1}$ 
      \State Set $\alpha_k \leftarrow \langle \vw_k,\vv_k \rangle $ and $\vw_k \leftarrow \vw_k-\alpha_k\vv_k$
      \State Set $\beta_{k+1} \leftarrow \|\vw_k\|$ and $\vv_{k+1}\leftarrow \vw_k/\beta_{k+1}$
      \EndFor \label{line:lanczos_end_minevec} \tikzmark{bottom}
    \State Form a tridiagonal matrix $\mT \leftarrow \mathsf{tridiag}(\beta_{2:N_1},\alpha_{1:N_1},\beta_{2:N_1})$
    \State \Comment{Use the tridiagonal structure to compute the minimum eigenvalue of $\mT$}\tikzmark{right}
    \State Compute $\bar{\lambda}_{\min} \leftarrow \mathsf{MinEig}(\mT)$ and set $\hat{\lambda}_{\min} \leftarrow \bar{\lambda}_{\min} - \frac{\delta}{2}$
    \If{$\hat{\lambda}_{\min} \geq 0$} \quad\Comment{Case (a)}
    \State Set $\hat{\vv}_{\min} \leftarrow 0 $ and return $(\hat{\lambda}_{\min},\hat{\vv}_{\min})$
    \Else \quad\Comment{Case (b)}
    \State Set $N_2 \leftarrow \lceil\frac{1}{4}\sqrt{\frac{2B}{\delta}}\log(\frac{44d B}{q^2 \delta})+\frac{1}{2} \rceil$
    \For{$k=N_1+1,\dots,N_2$ \tikzmark{top2}} \label{line:lanczos_start_minevec_2}
      \State Set $\vw_k \leftarrow \mA \vv_k-\beta_k \vv_{k-1}$ 
      \State Set $\alpha_k \leftarrow \langle \vw_k,\vv_k \rangle $ and $\vw_k \leftarrow \vw_k-\alpha_k\vv_k$
      \State Set $\beta_{k+1} \leftarrow \|\vw_k\|$ and $\vv_{k+1}\leftarrow \vw_k/\beta_{k+1}$
      \EndFor \tikzmark{bottom2} \label{line:lanczos_end_minevec_2}
    \State Form a tridiagonal matrix $\mT' \leftarrow \mathsf{tridiag}(\beta_{2:N_2},\alpha_{1:N_2},\beta_{2:N_2})$
    \State Form a pentadiagonal matrix $\mM \leftarrow (\mT' -\hat{\lambda}_{\min}\mI)^2 + \beta_{N_2+1}^2 \ve_{N_2}\ve_{N_2}^\top$
    \State \Comment{Use the pentadiagonal structure to compute the minimum eigenvector of $\mM$}
    \State Compute $\tilde{\vz}_{\min} \leftarrow \mathsf{MinEvec}(\mM)$ and set $\hat{\vv}_{\min} \leftarrow \sum_{k=1}^{N_2} \tilde{z}_{\min}^{(k)} \vv_k$
    \State Return $(\hat{\lambda}_{\min}, \hat{\vv}_{\min})$
    \EndIf%
  \end{algorithmic}%
  \AddNote{top}{bottom}{right}{\color{comment}\textit{\quad Lanczos iteration}}
  \AddNote{top2}{bottom2}{right}{\color{comment}\textit{\quad Lanczos iteration}}
\end{subroutine}

Now we are ready to describe the implementation of $\mathsf{MinEvec}$ in Subroutine~\ref{alg:lanczos}, which consists of two stages. 
In the first stage, we run the Lanczos method to obtain a good approximation of $\lambda_{\min}(\mA)$. Specifically, recall that $B$ is an upper bound on $\lambda_{\max}(\mA) - \lambda_{\min}(\mA)$. We initiate the process with a random vector $\vv_1$ uniformly drawn from the unit sphere and execute the Lanczos method for $N_1 =  \lceil\frac{1}{4}\sqrt{\frac{2B}{\delta}}\log(\frac{11d}{q^2})+\frac{1}{2} \rceil$ iterations (see Lines~\ref{line:lanczos_start_minevec} to~\ref{line:lanczos_end_minevec}). It is known that the Lanczos vectors $\{\vv_k\}_{k=1}^{N_1}$ form an orthonormal basis of the Krylov subspace $\mathcal{K}_{N_1}(\mA,\vv_1)$. Moreover, if we define $\mV^{(N_1)} = [\vv_1,\dots,\vv_{N_1}] \in \reals^{d \times N_1}$, then $(\mV^{(N_1)})^\top \mA \mV^{(N_1)}$ is a tridiagonal matrix $\mT \in \reals^{N_1 \times N_1}$ given by: 
$$\mT = \begin{bmatrix}
    \alpha_1 & \beta_2 & \\
    \beta_2  & \alpha_2 & \beta_3 \\ 
             & \beta_3  & \ddots & \ddots \\
             & & \ddots & \ddots & \beta_{N_1}\\ 
             & & & \beta_{N_1} & \alpha_{N_1} 
\end{bmatrix}.$$ 
Due to the tridiagonal structure, the eigenvalues of $\mT$ can be computed in $O(N_1)$ time. In particular, we compute its minimum eigenvalue $\bar{\lambda}_{\min}$, which satisfies:
\begin{equation}\label{eq:def_lambda_bar_min}
    \bar{\lambda}_{\min} = \min_{\vz \in \reals^{N_1}} \frac{\vz^\top \mT \vz}{\vz^\top \vz} = \min_{\vz \in \reals^{N_1}} \frac{\vz^\top (\mV^{(N_1)})^\top \mA \mV^{(N_1)} \vz}{\vz^\top \vz} = \min_{\vv \in \mathcal{K}_{N_1}(\mA,\vv_1)} \frac{\vv^\top \mA \vv}{\vv^\top \vv}. 
\end{equation}
Then we set $\hat{\lambda}_{\min} \leftarrow \bar{\lambda}_{\min} - \frac{\delta}{2}$. As we shall prove in Proposition~\ref{prop:minevec_correctness}, we have $\hat{\lambda}_{\min} \leq \lambda_{\min}(\mA) \leq  \hat{\lambda}_{\min}+ \frac{\delta}{2}$ with probability at least $1-\frac{q}{2}$.  

In the second stage, we compute the output $\hat{\vv}_{\min}$ in the definition of the $\mathsf{MinEvec}$ oracle. We distinguish two cases depending on the sign of $\hat{\lambda}_{\min}$. 
\begin{itemize}
    \item If $\hat{\lambda}_{\min} \geq 0$, we simply set $\hat{\vv}_{\min} = 0$ and return the pair $(\hat{\lambda}_{\min},\hat{\vv}_{\min})$. 
    \item Otherwise, if $\hat{\lambda}_{\min} < 0 $, we continue to run the Lanczos method for $N_2 = \lceil\frac{1}{4}\sqrt{\frac{2B}{\delta}}\log(\frac{44d B}{q^2 \delta})+\frac{1}{2} \rceil$ iterations (see Lines~\ref{line:lanczos_start_minevec_2} to~\ref{line:lanczos_end_minevec_2}). Define  $\mV^{(N_2)} = [\vv_1,\dots,\vv_{N_2}] \in \reals^{d \times N_2}$, then from the Lanczos iteration it holds that: 
    \begin{equation*}
        \mA \mV^{(N_2)} =  \mV^{(N_2)} \mT' + \beta_{N_2 +1 } \vv_{N_2+1} \ve_{N_2}^\top,
    \end{equation*}
    where $\mT'$ is a tridiagonal matrix given by $\mathsf{tridiag}(\beta_{2:N_2},\alpha_{1:N_2},\beta_{2:N_2})$ and $\ve_{N_2}$ is the $N_2$-th standard unit vector. Therefore, one can show that $(\mV^{(N_2)})^\top (\mA-\hat{\lambda}_{\min}\mI)^\top (\mA-\hat{\lambda}_{\min}\mI) \mV^{(N_2)}$ is a pentadiagonal matrix $\mM$ given by $(\mT' -\hat{\lambda}_{\min}\mI)^2 + \beta_{N_2+1}^2 \ve_{N_2}\ve_{N_2}^\top$. Similarly, we can compute the minimum eigenvector $ \tilde{\vz}_{\min}$ of $\mM$ in $\bigO(N_2)$ time. We further define $\hat{\vv}_{\min} = \mV^{(N_2)}\tilde{\vz}_{\min} $ and it satisfies:
    \begin{equation}\label{eq:def_hatv_min}
        \hat{\vv}_{\min} = \argmin_{\vv \in \mathcal{K}_{N_2}(\mA,\vv_1)} \frac{\vv^\top (\mA-\hat{\lambda}_{\min}\mI)^\top (\mA-\hat{\lambda}_{\min}\mI) \vv}{\vv^\top \vv} = \argmin_{\vv \in \mathcal{K}_{N_2}(\mA,\vv_1)} \frac{\|(\mA-\hat{\lambda}_{\min}\mI)\vv\|}{\|\vv\|}. 
    \end{equation}
    Finally, we return the pair $(\hat{\lambda}_{\min},\hat{\vv}_{\min})$.  
\end{itemize}

In the following proposition, we will prove that the pair $(\hat{\lambda}_{\min},\hat{\vv}_{\min})$ returned by Subroutine~\ref{alg:lanczos} satisfies the conditions specified in Definition~\ref{def:minevec}. 
\begin{proposition}\label{prop:minevec_correctness}
    With probability at least $1-q$, Subroutine~\ref{alg:lanczos} successfully implements the $\mathsf{MinEvec}$ oracle defined in Definition~\ref{def:minevec} and the total number of matrix-vector products is bounded by $\lceil\frac{1}{4}\sqrt{\frac{2B}{\delta}}\log(\frac{44d B}{q^2 \delta})+\frac{1}{2} \rceil$. 
\end{proposition}
\begin{proof}
    To begin with, we prove that $\hat{\lambda}_{\min} \leq \lambda_{\min}(\mA) \leq  \hat{\lambda}_{\min}+ \frac{\delta}{2}$ with probability at least $1-\frac{q}{2}$. By using the property of $\bar{\lambda}_{\min}$ in \eqref{eq:def_lambda_bar_min},we have $\bar{\lambda}_{\min} \geq \lambda_{\min}(\mA)$. Furthermore, by applying Lemma~\ref{prop:lanczos_estimate} with $\rho = \frac{\delta}{2B}$, we obtain that: 
    \begin{equation*}
    \Pr\Bigl(\bar{\lambda}_{\min} \geq \lambda_{\min}(\mA) + \frac{\delta}{2B} (\lambda_{\max}(\mA) - \lambda_{\min}(\mA))\Bigr) \leq 1.648 \sqrt{d} e^{-\sqrt{\frac{\delta}{2B}}(2N_1-1)} \leq \frac{q}{2}. 
    \end{equation*}
    Since $B \geq \lambda_{\max}(\mA) - \lambda_{\min}(\mA)$, this implies that, with probability at least $1-\frac{q}{2}$, we have $\bar{\lambda}_{\min} \leq \lambda_{\min}(\mA) + \frac{\delta}{2B} (\lambda_{\max}(\mA) - \lambda_{\min}(\mA)) \leq  \lambda_{\min}(\mA) + \frac{\delta}{2}$, leading to $\lambda_{\min}(\mA) \leq \bar{\lambda}_{\min} \leq  \lambda_{\min}(\mA) + \frac{\delta}{2}$. Since $\hat{\lambda}_{\min} = \bar{\lambda}_{\min} - \frac{\delta}{2}$, we get $\hat{\lambda}_{\min} \leq \lambda_{\min}(\mA) \leq  \hat{\lambda}_{\min}+ \frac{\delta}{2}$ with probability at least $1-\frac{q}{2}$. 

    Hence, in the first case where $\hat{\lambda}_{\min} \geq 0$, we have $ \lambda_{\min}(\mA) \geq \hat{\lambda}_{\min} \geq 0$ and thus the condition in Definition~\ref{def:minevec} is satisfied with probability at least $1-\frac{q}{2}$. In the second case where $\hat{\lambda}_{\min} <0$, it still holds that $\hat{\lambda}_{\min} + \frac{\delta}{2} \geq \lambda_{\min}(\mA) \geq \hat{\lambda}_{\min}$ with probability at least $1-\frac{q}{2}$. Moreover, using the property of $\hat{\vv}_{\min}$ in \eqref{eq:def_hatv_min}, by applying Lemma~\ref{lem:lanczos_new} with the matrix $\mA - \hat{\lambda}_{\min}\mI$ and $\rho = \frac{\delta}{2B}$, we obtain that:
    \begin{equation*}
    \Pr\Bigl(\|(\mA - \hat{\lambda}_{\min}\mI)\hat{\vv}_{\min}\| \geq \lambda_{\min}(\mA) - \hat{\lambda}_{\min} + {\textstyle\frac{\delta}{2B}} (\lambda_{\max}(\mA) - \lambda_{\min}(\mA))\Bigr) \leq 2.34 \sqrt{\frac{2Bd}{\delta}} e^{-\sqrt{\frac{\delta}{2B}}(2N_2-1)} \leq \frac{q}{2},
    \end{equation*}
    where we used $N_2 = \lceil\frac{1}{4}\sqrt{\frac{2B}{\delta}}\log(\frac{44d B}{q^2 \delta})+\frac{1}{2} \rceil$ in the last inequality. Using the union bound, with probability at least $1-q$, we have: 
    \begin{equation*}
        \hat{\lambda}_{\min} + \frac{\delta}{2} \geq \lambda_{\min}(\mA) \geq \hat{\lambda}_{\min} \; \text{and} \; \|(\mA - \hat{\lambda}_{\min}\mI)\hat{\vv}_{\min}\| \leq \lambda_{\min}(\mA) - \hat{\lambda}_{\min} + \frac{\delta}{2B} (\lambda_{\max}(\mA) - \lambda_{\min}(\mA)).  
    \end{equation*}
    Together, these two inequalities imply that $\|(\mA - \hat{\lambda}_{\min}\mI)\hat{\vv}_{\min}\| \leq \frac{\delta}{2} + \frac{\delta}{2} = \delta$. Hence, we conclude that all the conditions in Definition~\ref{def:minevec} are satisfied. 
\end{proof}

\subsection{Reducing Gradient Norm for Constrained Convex Optimization}\label{appen:gradient_norm}
Recall the general constrained problem in~\eqref{eq:constrained}. In this section, we describe the $\mathsf{FISTA\mathrm{+}SFG}$ algorithm in Proposition~\ref{prop:fista_sfg}, which consists of a total of $2N$ iterations. In the first $N$ iterations,  we run the FISTA algorithm proposed in~\cite{beck2009fast} with the initialization $\vx_0 = \vy_0  \in Q$ and $t_0 = 1$. It follows the following update: for any $k \in \{0,1,\dots,N-1\}$,  
\begin{equation*}
    \vx_{k+1} = \Pi_{Q}\Bigl(\vy_k - \frac{1}{L_g}\nabla g(\vy_k)\Bigr),\;t_{k+1} = \frac{1+\sqrt{1+4t_k^2}}{2},\;\vy_{k+1} = \vx_{k+1} + \frac{t_k-1}{t_{k+1}}(\vx_{k+1}-\vx_k). 
\end{equation*}
It is known that FISTA achieves the convergence rate: 
\begin{equation}\label{eq:fista_convergence}
    f(\vx_N) - f^* \leq \frac{2L_g\|\vx_0-\vx^*\|^2}{(N+1)^2}.
\end{equation} 
For the second $N$ iterations, we switch to the Super FISTA-G method from \cite{kim2023time}, initializing with $\tilde{\vx}_0 = \tilde\vy_0 = \vx_N$. The updates for Super FISTA-G are given by: 
\begin{align*}
    \tilde\vx_{k+1} &= \Pi_{Q}\Bigl(\tilde\vy_k - \frac{1}{4L_g}\nabla f(\tilde\vy_k)\Bigr),  \quad \forall k \in \{0,\dots,N-1\},\\
    \tilde\vy_{k+1} &=
    \begin{cases}
        \tilde\vx_{k+1} + {\textstyle\frac{(N-k)(2N-2k-3)}{(N-k+2)(2N-2k-1)}}(\tilde\vx_{k+1}-\tilde\vx_k) + {\textstyle\frac{(4N-4k-5)(2N-2k-3)}{6(N-k+2)(2N-2k-1)}}(\tilde\vx_{k+1}-\tilde\vy_k), & \text{if } k \in \{0,\dots,N-3\},\\
        \tilde\vx_{N-1} + \frac{3}{10}(\tilde\vx_{N-1} - \tilde\vx_{N-2}) + \frac{3}{40}(\tilde\vx_{N-1}- \tilde\vy_{N-2}), & \text{if } k = N-2.
    \end{cases}
\end{align*}
As shown in \cite{kim2023time}, this method achieves the following convergence bound:  
\begin{equation}\label{eq:SFG_convergence}
    \min_{\vu \in \mathcal{N}_Q(\tilde{\vx}_N)}\|\nabla g(\tilde{\vx}_N) + \vu\| \leq \sqrt{\frac{50 L_g (f(\tilde{\vx}_0) - f^*)}{(N+1)(N+2)}}.
\end{equation}
Combining \eqref{eq:fista_convergence} and~\eqref{eq:SFG_convergence}, we obtain  
\begin{equation*}
    \min_{\vu \in \mathcal{N}_Q(\tilde{\vx}_N)}\|\nabla g(\tilde{\vx}_N) + \vu\| \leq \sqrt{\frac{50 L_g (f({\vx}_N) - f^*)}{(N+1)(N+2)}} \leq \frac{10 L_g\|\vx_0-\vx^*\|}{(N+1)^2}.
\end{equation*}
Hence, to satisfy $ \min_{\vu \in \mathcal{N}_Q(\tilde{\vx}_N)}\|\nabla g(\tilde{\vx}_N) + \vu\| \leq \delta$, we can set $N = \sqrt{\frac{10L_g \|\vx_0-\vx^*\|}{\delta}}$ and thus the total number of gradient queries are bounded by  $2N =2\sqrt{\frac{10L_g \|\vx_0-\vx^*\|}{\delta}}$. This proves Proposition~\ref{prop:fista_sfg}. 

\section{Implementation of \texorpdfstring{$\mathsf{SEP}$}{SEP}}
\label{appen:SEP}

\begin{subroutine}[!t]\small
  \caption{$\mathsf{SEP}(\mW;q)$}\label{alg:SEP}
  \begin{algorithmic}[1]
      \State \textbf{Input:} $\mW \in \mathbb{S}^d$, $q\in (0,1)$
      \State \textbf{Initialize:} sample $\vv_1\in \reals^d$ uniformly from the unit sphere, $\beta_1 \leftarrow 0$, $\vv_0\leftarrow 0$
      \State Set the number of iterations 
      $N \leftarrow \Bigl\lceil \frac{1}{2}\log\frac{11d}{q^2}+\frac{1}{2}\Bigr\rceil$
      \For{$k=1,\dots,N$ \tikzmark{top}} \label{line:start of Lanczos}
      \State Set $\vw_k \leftarrow \mW \vv_k-\beta_k \vv_{k-1}$ 
      \State Set $\alpha_k \leftarrow \langle \vw_k,\vv_k \rangle $ and $\vw_k \leftarrow \vw_k-\alpha_k\vv_k$
      \State Set $\beta_{k+1} \leftarrow \|\vw_k\|$ and $\vv_{k+1}\leftarrow \vw_k/\beta_{k+1}$
      \EndFor \label{line:end of Lanczos}
    \State Form a tridiagonal matrix $\mT \leftarrow \mathsf{tridiag}(\beta_{2:N},\alpha_{1:N},\beta_{2:N})$ \label{line:tridiagonal}
    \State \Comment{Use the tridiagonal structure to compute eigenvectors of $\mT$}
    \State Compute $(\hat{\lambda}_1,\vz^{(1)}) \leftarrow \mathsf{MaxEvec}(\mT)$ and $(\hat{\lambda}_d, \vz^{(d)}) \leftarrow \mathsf{MinEvec}(\mT)$ \tikzmark{right}
    \State Set $\vu^{(1)} \leftarrow \sum_{k=1}^N z^{(1)}_k\vv_k$ and $\vu^{(d)} \leftarrow \sum_{k=1}^N z^{(d)}_k\vv_k$ \label{line:eigenvs}\tikzmark{bottom}
    \State Set $\gamma \leftarrow \max\{\hat{\lambda}_1,-\hat{\lambda}_d\}/L_1$ 
    \If{$\gamma \leq 1$}
    \State Return $\gamma$ and $\mS = 0$ \hspace{.75em}\Comment{Case I: $\gamma\leq 1$, which implies $\|\mW\|_{\op} \leq 2L_1$}
    \ElsIf{$\hat{\lambda}_1 \geq -\hat{\lambda}_d$}
      \State Return $\gamma$ and $\mS = \frac{1}{L_1}\vu^{(1)}(\vu^{(1)})^\top$ \hspace{.75em}\Comment{Case II: $\gamma> 1$ and $\mS$ defines a separating hyperplane}
    \Else
      \State Return $\gamma$ and $\mS = -\frac{1}{L_1}\vu^{(d)}(\vu^{(d)})^\top$ \hspace{.75em}\Comment{Case II: $\gamma> 1$ and $\mS$ defines a separating hyperplane}
    \EndIf
  \end{algorithmic}
  \AddNote{top}{bottom}{right}{\color{comment}\textit{\quad Lanczos method}}
\end{subroutine}

This section describes the implementation of the $\mathsf{SEP}$ oracle, as defined in Definition~\ref{def:extevec}. As outlined in Section~\ref{subsec:computational}, our approach relies on the Lanczos algorithm with a random start, and the procedure is presented in Subroutine~\ref{alg:SEP}.
Specifically, starting from a random vector $\vv_1 \in \reals^d$ uniformly drawn from the unit sphere, we execute the Lanczos method for $N$ iterations, where $N = \min\Bigl\{\Bigl\lceil \frac{1}{2}\log\frac{11d}{q^2}+\frac{1}{2}\Bigr\rceil, d\Bigr\}$ (see Lines~\ref{line:start of Lanczos} to~\ref{line:end of Lanczos}). It is known that the Lanczos vectors $\{\vv_k\}_{k=1}^N$ form an orthonormal basis of the Krylov subspace $\mathcal{K}_N(\mW,\vv_1) = \mathrm{span}\{\vv_1, \mW\vv_1, \dots,\mW^{N-1}\vv_1\}$. Using this basis, the matrix $\mW$  is represented in the Krylov subspace by a tridiagonal matrix $$\mT = \begin{bmatrix}
    \alpha_1 & \beta_2 & \\
    \beta_2  & \alpha_2 & \beta_3 \\ 
             & \beta_3  & \ddots & \ddots \\
             & & \ddots & \ddots & \beta_{N}\\ 
             & & & \beta_N & \alpha_N 
\end{bmatrix}.$$ 

Due to the tridiagonal structure, the eigenvectors of $\mT$ can be computed in $O(N)$ time. This computation yields two unit vectors, $\vu^{(1)}$ and $\vu^{(d)}$, such that
(see Lines~\ref{line:tridiagonal} to~\ref{line:eigenvs}): 
\begin{equation*}
    \vu^{(1)} = \argmax_{\vu \in \mathcal{K}_N(\mW,\vv_1)} \frac{\vu^\top \mW \vu}{\vu^\top \vu}, \quad \vu^{(d)} = \argmin_{\vu \in \mathcal{K}_N(\mW,\vv_1)} \frac{\vu^\top \mW \vu}{\vu^\top \vu}.
\end{equation*}
We then set $\gamma \leftarrow \max\{\hat{\lambda}_1,-\hat{\lambda}_d\}/L_1$, where $\hat{\lambda}_1 = (\vu^{(1)})^\top \mW \vu^{(1)}$ and $\hat{\lambda}_d = (\vu^{(d)})^\top \mW \vu^{(d)}$.  Now we distinguish two cases based on the value of $\gamma$. 
\begin{itemize}
    \item \textbf{Case I:} If $\gamma <1$, then we set $\mS = 0$ and return $(\gamma, \mS)$. 
    \item \textbf{Case II:} If $\gamma \geq 1$, we proceed with two subcases. If $\hat{\lambda}_1 \geq - \hat{\lambda}_d$, we set $\mS = \frac{1}{L_1} \vu^{(1)} (\vu^{(1)})^\top$. Otherwise, if $-\hat{\lambda}_d \geq \hat{\lambda}_1$, we set $\mS = -\frac{1}{L_1} \vu^{(d)} (\vu^{(d)})^\top$. 
\end{itemize}
In the next proposition, we will prove that the output $(\gamma,\mS)$ satisfy the conditions specified in Definition~\ref{def:extevec} with probability at least $1-q$. 
\begin{proposition}\label{prop:SEP}
With probability at least $1-q$, Subroutine~\ref{alg:SEP} successfully implements the $\mathsf{SEP}$ oracle and the total number of matrix-vector products is bounded by %
$\Bigl\lceil \frac{1}{2}\log\frac{11d}{q^2}+\frac{1}{2}\Bigr\rceil$.
\end{proposition}
\begin{proof}
First, we show that $\|\mW\|_{\op} \leq  2\gamma L_1$ holds with probability at least $1-q$. By using Proposition~\ref{prop:lanczos_estimate}, we have:
\begin{align*}
\Pr\Bigl(\hat{\lambda}_1 \leq \lambda_{\max}(\mW) - \frac{1}{4} (\lambda_{\max}(\mW) - \lambda_{\min}(\mW))\Bigr) &\leq 1.648 \sqrt{d} e^{-\frac{1}{2}(2N-1)} \leq \frac{q}{2}, \\
\Pr\Bigl(\hat{\lambda}_d \geq \lambda_{\min}(\mW) + \frac{1}{4} (\lambda_{\max}(\mW) - \lambda_{\min}(\mW))\Bigr) &\leq 1.648 \sqrt{d} e^{-\frac{1}{2}(2N-1)} \leq \frac{q}{2}.
\end{align*}
Hence, by using the union bound, with probability at least $1-q$, it holds that: 
\begin{equation*}
\hat{\lambda}_1 \geq \lambda_{\max}(\mW) - \frac{1}{4} (\lambda_{\max}(\mW) - \lambda_{\min}(\mW)), \quad 
    \hat{\lambda}_d \leq \lambda_{\min}(\mW) + \frac{1}{4} (\lambda_{\max}(\mW) - \lambda_{\min}(\mW)).
\end{equation*}
Combining these two inequalities yields $\hat{\lambda}_1 - \hat{\lambda}_d \geq \frac{1}{2}(\lambda_{\max}(\mW) - \lambda_{\min}(\mW))$, which implies that: 
\begin{align}
    \lambda_{\max}(\mW) &\leq \hat{\lambda}_1+ \frac{1}{4} (\lambda_{\max}(\mW) - \lambda_{\min}(\mW)) \leq \frac{3}{2}\hat{\lambda}_1 - \frac{1}{2}\hat{\lambda}_d, \label{eq:bound_on_max} \\
    \lambda_{\min}(\mW) &\geq \hat{\lambda}_d - \frac{1}{4} (\lambda_{\max}(\mW) - \lambda_{\min}(\mW)) \geq -\frac{1}{2}\hat{\lambda}_1 + \frac{3}{2} \hat{\lambda}_d. \label{eq:bound_on_min}
\end{align}
Recall that $\gamma = \max\{\hat{\lambda}_1,-\hat{\lambda}_d\}/L_1$, which means that $\max\{\hat{\lambda}_1,-\hat{\lambda}_d\} = \gamma L_1 $. Hence, \eqref{eq:bound_on_max} and \eqref{eq:bound_on_min} further lead to $\lambda_{\max}(\mW) \leq 2\gamma L_1$ and $\lambda_{\min}(\mW) \geq -2\gamma L_1$. Since $\|\mW\|_{\op} = \max\{\lambda_{\max}(\mW),-\lambda_{\min}(\mW)\}$, this further implies that $\|\mW\|_{\op} \leq  2\gamma L_1$.

\looseness = -1
In the following, we assume that $\|\mW\|_{\op} \leq  2\gamma L_1$, which happens with probability $1-q$.
Thus, in \textbf{Case I} where $\gamma <1$, we have $\|\mW\|_{\op} \leq 2L_1$ and  the condition in Definition~\ref{def:extevec} is satisfied  Otherwise, in \textbf{Case II} where $\gamma \geq 1$, note that $\|\mW/\gamma\|_{\op} = \|\mW\|_{\op}/\gamma \leq 2L_1$. Moreover, without loss of generality, assume that $\hat{\lambda}_1 \geq -\hat{\lambda}_d$; we can use the argument when $\hat{\lambda}_1 \leq -\hat{\lambda}_d$. In this case, we have $\gamma = \hat{\lambda}_1/L_1$ and $\mS = \frac{1}{L_1} \vu^{(1)} (\vu^{(1)})^\top$. Since $\vu^{(1)}$ is a unit vector, it is easy to verify that $\|\mS\|_{F} = \frac{1}{L_1}$. Moreover, for any $\mB \in \sS^d$ such that $\|\mB\|_{\op} \leq L_1$, it holds that: 
\begin{equation*}
    \langle \mS, \mW - \mB \rangle  = \frac{1}{L_1} (\vu^{(1)})^\top \mW \vu^{(1)} - \frac{1}{L_1} (\vu^{(1)})^\top \mB \vu^{(1)}. 
\end{equation*}
Note that $\frac{1}{L_1} (\vu^{(1)})^\top \mW \vu^{(1)} = \frac{\hat{\lambda}_1}{L_1} = \gamma$ and $\frac{1}{L_1} (\vu^{(1)})^\top \mB \vu^{(1)} \leq \frac{1}{L_1} \|\mB\|_{\op} \|\vu^{(1)}\| \leq 1$. Thus, we obtain that $\langle \mS, \mW - \mB \rangle \geq \gamma -1$. Hence, we conclude that all 
the conditions in Definition~\ref{def:extevec} are satisfied. 
\end{proof}

\section{Proof of Theorem \texorpdfstring{\ref{thm:computational}}{4.6}}\label{appen:computational}
Having established Corollary~\ref{coro:TRSolver} and Proposition~\ref{prop:SEP}, we proceed to analyze the total computational cost of our proposed Algorithm~\ref{alg:conversion}.
Recall that in each iteration, we have one call to the $\mathsf{TRSolver}$ oracle (Line~\ref{line:TRcall} in Algorithm~\ref{alg:conversion}) and one call to the $\mathsf{SEP}$ oracle (Line~\ref{line:SEPcall} in Subroutine~\ref{alg:hessian_approx}). 
Moreover, by Theorem~\ref{thm:convergence_rate}, we can find an $\epsilon$-FOSP after at most $M = \mathcal{O}(d^{1/4}/\varepsilon^{13/8})$ iterations. 

First, for a given failure probability $p \in (0,1)$, we can pick $q = \frac{p}{2M}$ so that, by the union bound, all of our calls to the $\mathsf{TRSolver}$ oracle and the $\mathsf{SEP}$ oracle (Line~\ref{line:SEPcall}) are successful with probability at least $1-p$. 
Now we consider the computational cost of $\mathsf{TRSolver}$. Recall that the input matrix is given by $\mA_n = \frac{1}{2}\mB_n + \frac{1}{\eta}\mI$ and we set $\delta = \frac{D}{\eta T}$ in Theorem~\ref{thm:convergence_rate}. Moreover, since $\|\mB_n\| \leq 2L_1$ (see Lemma~\ref{lem:surrogate_regret}), we have $\lambda_{\max}(\mA_n) - \lambda_{\min}(\mA_n) \leq \frac{1}{2}(\lambda_{\max}(\mB_n) - \lambda_{\min}(\mB_n)) \leq 2L_1$ and $\lambda_{\max}(\mA_n) \leq L_1+\frac{1}{\eta}$.  Hence, it follows from Corollary~\ref{coro:TRSolver} that the number of matrix-vector products for each call of $\mathsf{TRSolver}$ is bounded by:
\begin{equation*}
    \biggl\lceil\frac{1}{2}\sqrt{2\eta L_1 T }\log(\frac{704d L_1 \eta T M^2}{p^2 })+\frac{1}{2} \biggr\rceil + \max\biggl\{2\sqrt{{10 (\eta L_1 +1)T}},  4\sqrt{{10 \eta L_1 T}+10}\biggl\} = \tilde{\bigO}(\sqrt{\eta L_1 T} + \sqrt{T}).
\end{equation*}
Hence, after $M$ iterations, the total number of matrix-vector products is given by 
\begin{equation*}
    \tilde{\bigO}(\sqrt{\eta L_1 T} M + \sqrt{T} M) = \tilde{\bigO} \left( \frac{d^{1/8}L_1^{5/8} L_2^{3/16} (f(\vx_0)-f^*)}{\epsilon^{29/16}} + \frac{d^{3/8} L_1^{3/8}L_2^{5/16}(f(\vx_0)-f^*)}{\epsilon^{27/16}}\right)
\end{equation*}
where we used the expression of $\eta$ and $T$ in \eqref{eq:parameters_cr_exp} and $M = \mathcal{O}(d^{1/4}/\varepsilon^{13/8})$.
Furthermore, regarding the computational cost of the $\mathsf{SEP}$ oracle, it follows from Proposition~\ref{prop:SEP} that the total number of matrix-vector products is given by 
\begin{equation*}
    \Bigl\lceil \frac{1}{2}\log\frac{44dM^2}{p^2}+\frac{1}{2}\Bigr\rceil M = \tilde{\bigO} \left( \frac{d^{1/4} L_1^{1/4}L_2^{3/8} (f(\vx_0)-f^*)}{\epsilon^{13/8}}\right). 
\end{equation*}
Combining the two results, we complete the proof.